\documentclass[12pt]{amsart}
\usepackage{amssymb,amsmath,amsthm,epsfig,graphicx}
\usepackage[usenames]{color}
\def\Z{\mathbb{Z}}
\def\R{\mathbb{R}}
\def\Q{\mathbb{Q}}
\def\S{\mathbb{S}}
\def\E{\mathcal{E}}
\def\V{\mathcal{V}}
\def\B{\mathcal{B}}
\def\F{\mathcal{F}}
\def\RR{\mathcal{R}}

\newtheorem{theo}{\bf{Theorem}}[section]
\newtheorem{lem}[theo]{Lemma}
\newtheorem{cor}[theo]{Corollary}
\newtheorem{prop}[theo]{Proposition}
\newtheorem{prob}[theo]{Problem}
\newtheorem{conj}[theo]{Conjecture}
\theoremstyle{definition}
\newtheorem{defi}[theo]{Definition}
\newtheorem{rem}[theo]{Remark}
\newtheorem{exa}[theo]{Example}

\DeclareMathOperator{\rank}{rank}
\DeclareMathOperator{\lk}{lk}
\DeclareMathOperator{\antist}{ast}
\DeclareMathOperator{\GL}{GL}
\DeclareMathOperator{\Span}{Span}
\DeclareMathOperator{\lex}{lex}
\DeclareMathOperator{\Lker}{Lker}
\DeclareMathOperator{\row}{row}

\begin{document}
\title{Bipartite Rigidity}
\author{Gil Kalai}
\address{{Institute of Mathematics, Hebrew University of Jerusalem,
Jerusalem 91904, Israel}
{\tiny and}
{Department of Computer Science and  Department of Mathematics,
Yale University, New Haven, CT 06511, USA}}
\email{kalai@math.huji.ac.il}

\author{Eran Nevo}
\address{
Department of Mathematics,
Ben Gurion University of the Negev,
Be'er Sheva 84105, Israel
}
\email{nevoe@math.bgu.ac.il}

\author{Isabella Novik}
\address{
Department of Mathematics, Box 354350,
University of Washington, Seattle, WA 98195-4350, USA
}
\email{novik@math.washington.edu}

\thanks{Research of the first author was partially supported by
ERC advanced grant 320924, ISF grant 768/12, and NSF grant DMS-1300120,
of the second author by Marie Curie grant IRG-270923 and ISF grant 805/11,
and of the third author by NSF grant DMS-1069298.
The first author also acknowledges the Simons Institute for the Theory of Computing.}

\begin{abstract}
We develop a bipartite rigidity theory for bipartite
graphs parallel to the classical  rigidity theory for general graphs,
and define for two positive integers $k,l$ the notions of
$(k,l)$-rigid and $(k,l)$-stress free bipartite graphs.
This theory coincides with the study of Babson--Novik's balanced
shifting restricted to graphs. We establish bipartite analogs of
the cone, contraction, deletion, and gluing lemmas, and apply these
results to derive a bipartite analog of the rigidity criterion for planar
graphs. Our result asserts that for a planar bipartite graph $G$
its balanced shifting, $G^b$, does not contain $K_{3,3}$; equivalently,
planar bipartite graphs are generically $(2,2)$-stress free.
We also discuss potential applications of this
theory to Jockusch's cubical lower bound conjecture and to
upper bound conjectures for embedded simplicial complexes.
\end{abstract}

\maketitle

\section{Introduction}\label{sec:Intro}
\subsection{Three basic properties of planar graphs}

We start with three important results on planar graphs
and a motivating conjecture in a higher dimension:

\begin{prop}[Euler, Descartes]\label{prop:EulerIneq}
A simple planar graph with $n\geq 3$ vertices has at most $3n-6$ edges.
\end{prop}

\begin{prop}[Wagner, Kuratowski (easy part)]
A planar graph does not contain $K_5$ and $K_{3,3}$ as minors.
\end{prop}

The first result (that can be traced back to Descartes) is a
simple consequence of Euler's theorem. The second result is the
easy part of Wagner's characterization of planar graphs \cite{Wagner:Kuratowski}
which asserts that not having $K_5$ and $K_{3,3}$ as minors
characterizes planarity.

We will now state a third fundamental result on planar graphs.
This requires some definitions. An embedding of a graph $G$ into $\R^d$
is a map assigning a vector $\phi (v)\in\R^d$ to every vertex $v$.
The embedding is {\em stress-free} if there is no way
to assign weights $w_{uv}$ to edges, so that not all weights
are equal to zero and every vertex is ``in equilibrium'':
\begin{equation}\label{eq:st-free}
\sum_{v \ :\  uv \in E(G)}
w_{uv} (\phi (u)-\phi (v)) = 0 \quad \mbox{for all } u.
\end{equation}

\noindent The embedding is \emph{infinitesimally rigid} if every assignment
of velocity vectors $V(u)\in\R^d$ to vertices of $G$ that satisfies
\begin{equation} \label {e:iv}
\langle V(v)-V(u),\phi(v)-\phi(u)\rangle=0
\end{equation}
 for every $uv \in E(G),$ must satisfy relation (\ref{e:iv})
for {\em every pair} of vertices.\footnote{Relation (\ref{e:iv})
asserts that the velocities respect (infinitesimally) the distance
along an embedded edge. If these relations apply to all pairs of
vertices the velocities necessarily come\emph{}
 from a rigid motion of the entire space.}

\begin{prop}[Gluck, Dehn, Alexandrov, Cauchy]   \label{Prop1.3:Gluck}
A generic embedding of a simple planar graph in $\R^3$
is stress free. A generic embedding of a maximal simple planar
graph in $\R^3$ is also infinitesimally rigid.
\end{prop}

This result of Gluck \cite{Gluck} is closely related to Cauchy's rigidity
theorem for polytopes of dimension three, and is easily derived from
its infinitesimal counterpart by Dehn and Alexandrov.
We refer our readers to \cite{Connelly:RigiditySurvey, Pak-book}
for an exposition and further references. The above three
results on planar graphs also have interesting inter-connections.

In this paper we consider extensions of these three results
to bipartite graphs. The extensions to bipartite graphs are of
interest on their own and they are also offered as an approach
toward hard higher-dimensional generalizations such as the
following conjecture that in a slightly different
form was raised as a question by Gr\"unbaum \cite[Section 3.7]{Grunbaum-70}.
For a simplicial complex $K$ let $f_i(K)$ denote the number of
$i$-dimensional faces of $K$.

\begin{conj}\label{c:ksw}
There is an absolute constant $C$ such that a $2$-dimensional
simplicial complex $K$ embedded in $\R^4$ satisfies
\[ f_2(K) \le C f_1(K). \]
\end{conj}

\subsection{Bipartite rigidity and a bipartite analog of Gluck's theorem}

We develop a bipartite analog for the rigidity theory of graphs.

A {\em $(k,l)$-embedding} of a bipartite
graph $G=(A\uplus B, E)$ is a map $\phi: A\uplus B \to \R^{k}\times \R^{l}$
that assigns to every $a\in A$ a vector $\phi(a)\in\R^k\times(0)$,
and to every $b\in B$ a vector $\phi(b)\in (0)\times \R^l$.
A $(k,l)$-embedding $\phi$ of a bipartite graph $G=(A\uplus B, E)$
is {\em $(k,l)$-stress free} if there is no way to assign weights $w_{ab}$ to edges
so that not all weights are equal to zero and every vertex $u$
satisfies:
\begin{equation}  \label{e:iii-bip}
\sum_{v \ : \ uv\in E} w_{uv}\phi(v) =0.
\end{equation}
A $(k,l)$-embedding $\phi$ of a bipartite graph $G$ is {\em $(k,l)$-rigid}
if every assignment of velocity vectors
$V(a)\in (0)\times \R^l$ for $a\in A$ and
$V(b)\in\R^k\times (0)$ for $b\in B$ that satisfies
\begin{equation}  \label{e:iv-bip}
\langle V(a), \phi(b)\rangle + \langle V(b), \phi(a)\rangle =0
\end{equation}
for all $ab\in E$, must satisfy equation (\ref{e:iv-bip}) for all
$ab\in A\times B$.

It is worth mentioning (see Remark \ref{rem:Hyperconnectivity}
and Theorem \ref{thm:(l,1)-Laman})
that $(k,k)$-rigidity is equivalent to Kalai's
hyperconnectivity \cite{Kalai-hyperconnectivity} (restricted to
the bipartite case), while $(k,1)$-stress
freeness can be traced to the work of Whiteley \cite{Whiteley-scenes}.
In addition, $(k,k)$-rigidity is also equivalent to
the Singer and Cucuringu's notion of {\em rectangular local
completability in dimension $k$} as defined
in \cite[Section 4]{SingerCucuringu} in relation
to the problem of completing a low-rank matrix from a subset
of its entries.

In Section 4 we prove the following bipartite analog of Gluck's theorem.
\begin{theo}\label{thm:IntroBipGluck}
A generic $(2,2)$-embedding of a simple planar bipartite graph is
$(2,2)$-stress free. A generic $(2,2)$-embedding of a maximal simple
planar bipartite graph is also $(2,2)$-rigid.
\end{theo}

Our theory of bipartite rigidity relies on the notion of
``balanced shifting" that we sketch below.

%%%%%%%%%%%%%%%%%%
\subsection{Shifting, balanced shifting, and bipartite graphs}

Algebraic shifting is an operation introduced by
Kalai \cite{Kalai:fVectorsFamiliesConvexSetsI-84, Kalai:Diameter-91, Kalai-AlgShifting}
that replaces a simplicial complex $K$ with a ``shifted'' simplicial
complex $\Delta(K)$. There are two versions of algebraic shifting:
the symmetric one and the exterior one; we write
$\Delta(K)=K^s$ in the former case, and $\Delta(K)=K^e$ in the latter case.
For graphs, shifting is closely related to infinitesimal rigidity.
The shifting operation preserves various properties of the
complex, and, in particular, the numbers of faces of every dimension.
In dimension one, shifted graphs are known as threshold graphs:
the vertices numbered $\{1,2,\dots ,n\}$ are assigned nonnegative weights
$w_1>w_2>w_3> \cdots > w_n$, and edges correspond to pairs of
vertices with sum of weights above a certain threshold.\footnote{In higher
dimensions the class of shifted complexes
is much reacher than the class of threshold complexes.}

The following result (see \cite{Kalai-AlgShifting, Nevo-Stresses})
expresses Gluck's theorem (Proposition \ref{Prop1.3:Gluck})
in terms of symmetric shifting,
and clearly implies Euler's inequality of Proposition \ref{prop:EulerIneq}:

\begin{prop}\label{prop:K_5}
%\item[1.4]
If $G$ is a planar graph then the symmetric algebraic shifting of $G$,
$G^s$, does not contain $K_5$ as a subgraph.
Equivalently, $G^s$ does not contain the edge $\{4,5\}$.\footnote {A
shifted graph contains $K_5$ as a subgraph if and only if
 it contains it as a minor.}
More generally, the same statement holds for every graph $G$
that does not contain $K_5$ as a minor.
\end{prop}

\noindent One drawback of this result is that $G^s$ may contain
$K_{3,3}$ and hence the planarity property itself is lost under
shifting.

Similarly, the following conjecture implies Conjecture \ref{c:ksw}
with the sharp constant $C=4$:

\begin{conj} \label{2-dimKS}
If $K$ is a 2-dimensional simplicial complex embeddable in $\R^4$ then
$K^s$ does not contain the $2$-face $\{5,6,7\}$.
\end{conj}

We now move from graphs to {\em bipartite} graphs and, more
generally, in higher dimensions from simplicial complexes to
{\em balanced} simplicial complexes. A $d$-dimensional simplicial complex
is called \emph{balanced} if its vertices are colored with $d+1$
colors in such a way that every edge is bicolored; thus
the $d+1$ colors
for the vertices of every $d$-simplex are all different.
A balanced $1$-dimensional complex is simply a bipartite graph.
The study of enumerative and algebraic properties of
balanced complexes was initiated by Stanley \cite{Stanley-balancedCM}.

Babson and Novik \cite {Babson-Novik} defined a notion of balanced shifting,
and associated with every balanced simplicial complex $K$ a
balanced-shifted complex $K^b$. We recall this operation in Section 2
(see also Section 7) mainly concentrating on the case of graphs.
In Section 3, we show that in this case the properties of the
balanced-shifted bipartite graphs are described in terms
of ``bipartite rigidity'' as defined in Section 1.2. Specifically, we establish
bipartite analogs of the cone, contraction, deletion, and gluing lemmas;
then in Section 4 we use these results to prove Theorem \ref{thm:IntroBipGluck}
expressed in terms of balanced shifting as follows:

\begin{theo} \label{thm:planar-1}
For a bipartite planar graph $G$, $G^b$ does not contain $K_{3,3}$.
\end{theo}

A balanced-shifted bipartite graph without $K_{3,3}$ is planar, and therefore
Theorem \ref{thm:planar-1} implies that, in contrast with the case of symmetric
shifting, the planarity property {\em is preserved} under balanced shifting.
In other words, Theorem \ref{thm:planar-1} settles the $d=1$ case of the
following conjecture.

\begin{conj} \label{conj:embed-1}
Balanced shifting for $d$-dimensional balanced complexes
preserves embeddability in $\R^{2d}$.
\end{conj}

\noindent
In Section 8 we discuss several variations as well as
a more detailed version of this conjecture. It is also
worth remarking that the $d=2$ case of Conjecture \ref{conj:embed-1}
implies Conjecture \ref{c:ksw} with $C=9$, see Section 8
for more details.

We also discuss (see Section 6) a rigidity approach and
some partial results regarding the following conjecture of
Jockusch \cite{Jock}:

\begin{conj}[Jockusch] \label{conj:Adin}
If $K$ is a cubical polytope of dimension $d\geq 3$,
with $V$ vertices and $E$ edges then $E\geq \frac{d+1}{2}V-2^{d-1}$.
%(Equivalently, $g^{c}_2(K)\geq 0$.)
\end{conj}

The structure of the rest of the paper is as follows. In Section 2
we discuss basics of graphs as well as recall how to compute the
balanced shifting, $G^b$, of a bipartite graph $G$. In Section 3, we
define the notions of bipartite rigidity and stress freeness, and establish
bipartite analogs of the cone, deletion, contraction, and gluing
lemmas. In Section 4, we use these lemmas to prove
Theorem \ref{thm:planar-1}; we also discuss there balanced shifting
of linklessly embeddable graphs.
In Section 5 we consider bipartite analogs of Laman's theorem; this includes
analyzing balanced shifting of bipartite trees and outerplanar graphs.
Section 6 is devoted to graphs of cubical polytopes (and, more specifically,
to Jockusch's conjecture) as well as to graphs of polytopes that are
dual to balanced simplicial polytopes. In the remaining sections we
turn to higher-dimensional simplicial complexes: in Section 7 we recall basics
of simplicial complexes, and in Section 8 we discusses several problems and
partial results related to Conjecture \ref{conj:embed-1}.

Our rigidity theory of bipartite graphs has also led
to purely graph-theoretic questions regarding bipartite
graphs that we study separately in a joint work
with Chudnovsky and Seymour \cite{Chud.Kal.Nev.Nov.Sey.-BipMinors}.
Understanding the higher-dimensional analogs of graph minors for
general complexes has been quite fruitful in establishing partial
results on Conjecture \ref{c:ksw},
see \cite{Nevo:HigherMinors, Wagner11:HomologicalMinor}.
Finding such a notion for balanced complexes
may be equally useful, and bipartite graphs are a good place to start.
In \cite{Chud.Kal.Nev.Nov.Sey.-BipMinors} we initiate this program by defining
a notion of bipartite minors and proving a bipartite analog of Wagner's
theorem: a bipartite graph is planar
if and only if it does not have $K_{3,3}$ as a bipartite minor.

%%%%%%%%%%%%%%%%%%%%%%%%%%%%%%

\section{Preliminaries on bipartite graphs and
balanced shifting}\label{sec:PrelimBipGraphsBalShift}

All the graphs considered in this paper are simple graphs. A graph
with the vertex set $V$ and the edge set $E$ is denoted by $G=(V,E)$. A graph
is {\em bipartite} if there exists a bipartition of the vertex set $V$
of $G$, $V=A\uplus B$, in such a way that no two vertices from the
same part form an edge. When discussing bipartite graphs,
we fix such a bipartition and write $G=(A\uplus B,E)$; we refer to $A$
and $B$ as {\em parts} or {\em sides} of $G$.

If $G=(V,E)$ is a graph and $v$ is a vertex of $G$, then $G-v$ denotes
the induced subgraph of $G$ on the vertex set $V-\{v\}$.
If $G=(V=A\uplus B,E)$ is a bipartite graph and $u,v$ are two vertices
from the \emph{same} part, then the {\em contraction} of $u$ with $v$ is the graph
$G'$ on the vertex set $V-\{u\}$ obtained from $G$ by identifying $u$ with
$v$ and deleting the extra copy from each double edge that was created.
Observe that $G'$ is also bipartite.

We usually identify $A$ and $B$ with the ordered sets
$\{1<2<\ldots<n\}:=[n]$ and $\{1'<2'<...<m'\}:=[m']$, respectively.
For brevity, we denote the edge connecting vertices $i$ and $j'$ by $ij'$
(instead of $\{i,j'\}$). Define $\E=\E_{n,m}:=\{ij' : i\in[n], j'\in[m']\}$
to be the edge set of the complete bipartite graph $K_{A,B}=K_{n,m}$ on $V=A\uplus
B$. We also consider a total order, $<$, on $V$ that extends the given orders
on $A$ and $B$, and the induced lexicographic order, $<_{\lex}$, on $\E$.

Given a bipartite graph $G$ and such an order $<$ on $V$,
one can compute the {\em balanced shifting} of $G$, $G^b=G^{b,<}$.
This notion was introduced in \cite{Babson-Novik} for a much more general
class of simplicial complexes (and was called ``colored shifting'' there).
For the sake of completeness and to establish notation, we briefly recall
here the relevant definitions.

Let $G=(A\uplus B, E)$ be a bipartite graph, and
let $\R$ be the field of real numbers
(although all the theory we develop  here works over any infinite
field). Consider two sets of variables: $\{x_1,\ldots,x_n\}$
(one $x$ for each element of $A$) and $\{y_1,\ldots,y_m\}$
(one $y$ for each element of $B$). Let $S$ be a polynomial ring over $\R$
in the $x$'s and $y$'s, let $I_G$ be the Stanley-Reisner ideal of $G$:
\[
I_G=\langle\{x_ix_j : 1\leq i<j\leq n\}\cup \{y_iy_j:  1\leq i<j\leq m\}
\cup\{x_iy_j:ij'\notin E\}\rangle,
\]
and let $\R[G]:=S/I_G$ be the Stanley-Reisner ring of $G$. Setting
$\deg x_i:=(1,0)$ for all $i\in[n]$ and $\deg y_j:=(0,1)$ for all $j\in[m]$
makes $R[G]$ into a $\Z^2$-graded ring. For $(p,q)\in\Z^2$ we denote
by $\R[G]_{(p,q)}$ the $(p,q)$-th homogeneous component of $R[G]$.

Let $\Theta\in \GL_n(\R)\times \GL_m(\R)$ be a matrix
whose entries $\theta_{ij}, \theta_{s't'} \in \R$ for
$i, j\in [n]$ and $s',t'\in[m']$ are ``generic'' (for
instance, algebraically independent over $\Q$ is more than enough).
Set $\theta_i:=\sum_{j=1}^n \theta_{ij}x_j$ (for $1\leq i \leq n$) and
$\theta_{s'}:=\sum_{t=1}^m \theta_{s't'}y_t$ (for $1'\leq s'\leq m'$).
The matrix $\Theta$ acts on the set of linear forms of $\R[G]$ by
$\Theta x_i:=\theta_i$ and $\Theta y_s:=\theta_{s'}$, and this action
can be extended uniquely to a ($\Z^2$-grading preserving) ring
automorphism of $\R[G]$ that we also denote by $\Theta$.

Given a total order $<$ on $A\uplus B$, define $G^b=G^{b,<}$ --- the
{\em balanced shifting} of $G$ --- as the bipartite graph whose vertex set
is $A\uplus B$ and whose edge set, $E^b$, is given by
\[
\left\{ij'\in\E \ : \  \theta_i\theta_{j'} \notin
\Span\{ \theta_p\theta_{q'} : pq'<_{\lex} ij' \} \subset \R[G]_{(1,1)}
\right\}.
\]
In other words, the edge set of $G^b$ is determined by the ``greedy''
lexicographic basis of the vector space $\R[G]_{(1,1)}$ chosen from the
monomials written in $\theta$'s. For instance, $(K_{n,m})^{b,<}=K_{n,m}$
for any order $<$.

The following two properties of the balanced shifting from
\cite{Babson-Novik}
will be handy:
\begin{lem} For a bipartite graph $G=(A\uplus B, E)$ and any order $<$ on
$A\uplus B$ that extends the natural orders on $A$ and $B$, we have
\begin{itemize}
\item $|E^b|=|E|$, and
\item $G^b$ is \emph{balanced-shifted}: if $ij'\in E^b$, $1\leq p\leq i$, and
$1'\leq q'\leq j'$, then $pq'\in E^b$.
\end{itemize}
\end{lem}

We finish this section with the following definition and observation
that will be useful in the rest of the paper.

\begin{defi}
Given a pair of two fixed integers $k\leq n$ and $l\leq m$, we say
that a total order on $V=A\uplus B$ is \emph{$(k,l)$-admissible}
if (i) it extends the natural orders on $A$ and $B$, and (ii) the set
$[k]\cup[l']$ forms an initial segment of $V$ w.r.t.~$<$.
\end{defi}

\begin{lem}\label{mapPhi}
Let $G$ be a bipartite graph, let $<$ be a $(k,l)$-admissible order, and
let $\Phi=\Phi_G^{(k,l)}:
\left(\R[G]_{(1,0)}\right)^l \oplus \left(\R[G]_{(0,1)}\right)^k
\longrightarrow \R[G]_{(1,1)}$ be the following linear map:
\[(f_1, f_2,\ldots, f_l, g_1,\ldots, g_k)\mapsto
\sum_{i=1}^l \theta_{i'}f_i + \sum_{j=1}^k \theta_{j}g_j. \]
Then \begin{enumerate}
\item[1.]
all elements of $\E^{kl}:=\{ij'\in\E : i\leq k \mbox{ or } j'\leq l'\}$
are edges of $G^{b,<}$ if and only if the dimension of the image of
$\Phi$ equals $ln+km-kl$;
\item[2.]
the pair $(k+1)(l+1)'$ is not an edge of $G^{b,<}$ if and only if $\Phi$
is surjective.
\end{enumerate}
\end{lem}
\begin{proof} Since $\{\theta_1,\ldots,\theta_n\}$ is a basis of
$R[G]_{(1,0)}$ and $\{\theta_{1'},\ldots,\theta_{m'}\}$ is
a basis of $R[G]_{(0,1)}$, the set
$\{\theta_i\theta_{j'} : ij'\in \E^{kl}\}$ is a spanning set of the
image of $\Phi$. On the other hand, by $(k,l)$-admissibility of the
order $<$, $\E^{kl}$ is an initial segment
of $\E$ w.r.t~$<_{\lex}$. Hence, by the definition of $G^{b,<}$,
$\E^{kl}\subseteq E^b$ if and only if
$\{\theta_i\theta_{j'} : ij'\in \E^{kl}\}$ is a linearly independent
subset of $\R[G]_{(1,1)}$. Therefore, $\E^{kl}\subseteq E^b$ if and only
if $\{\theta_i\theta_{j'} : ij'\in \E^{kl}\}$ is a basis of the image of
$\Phi$. Part 1 follows.

The reasoning for Part 2 is similar:
since $G^b$ is balanced-shifted, the pair  $(k+1)(l+1)'$ is not
an edge of $G^{b}$ if and only if $E^b\subseteq \E^{kl}$.
Further, the fact that $\E^{kl}$ is an initial segment of $\E$
w.r.t~$<_{\lex}$ yields that $E^b\subseteq \E^{kl}$ if and only if
$\{\theta_i\theta_{j'} : ij'\in \E^{kl}\}$ is a spanning set of
$R[G]_{(1,1)}$, which implies Part 2. \end{proof}

It is worth noting that since $G^b$ is balanced-shifted,
the edge $(k+1)(l+1)'$ is not an edge of $G^b$ if and only if
$G^b$ does not contain $K_{k+1,l+1}$ as a subgraph.

%%%%%%%%%%%%%%%%%%%%%%%%%%%%%%%
%%%%%%%%%%%%%%%%%%%%%%%%%%%%%%%%%%%

\section{$(k,l)$-rigidity}\label{sec:(k,l)-rigidity}
The goal of this section is to develop a rigidity theory
for bipartite graphs, paralleling the one for general graphs
\cite{Asi-Roth1,Asi-Roth2,Whiteley-Cone,Wh-VertexSplitting}.
We recall from \cite[Section 2.7]{Kalai-AlgShifting} and \cite{Lee}, as discussed in detail in
\cite[Section 3.2]{Nevo-PhDthesis}, that a
(non-bipartite) graph $G$ on the vertex set $[n]$ is
{\em generically $d$-stress free} if and only if the pair
$(d+1)(d+2)$ is {\em not} an edge of the symmetric shifting of
$G$, $G^s$, and that $G$ is {\em generically $d$-rigid}
if and only if the pair $dn$ {\em is} an edge of $G^s$.
Motivated by these results, we make the following definition.
We use the same notation as in the previous section.

\begin{defi} \label{comb-rig}
 Let $G=(A\uplus B,E)$ be a bipartite graph,
let $k\leq n$ and $l\leq m$ be two fixed integers, and let $<$
be a $(k,l)$-admissible order on $A\uplus B$. We call $G$
 (generically) \emph{$(k,l)$-stress free} if the pair $(k+1)(l+1)'$
is not an edge of $G^{b,<}$. We say that  $G$ is (generically)
\emph{$(k,l)$-rigid} if all pairs $ij'\in\E$ such that $i\leq k$ or
$j'\leq l'$ are edges of $G^{b,<}$ .
\end{defi}

It follows from Lemma \ref{mapPhi} that being $(k,l)$-rigid
($(k,l)$-stress free, respectively) does not depend
on a particular choice of a $(k,l)$-admissible order $<$.
In fact, writing the matrix of the map $\Phi_G^{(k,l)}$ from
Lemma \ref{mapPhi} with respect to the basis
$(x_iy_j : ij'\in E)$ of $\R[G]_{(1,1)}$ and the basis
\[(l \mbox{ copies of } x_1, l \mbox{ copies of } x_2,
\ldots, l\mbox{ copies of } x_n,
k \mbox{ copies of } y_1, \ldots, k \mbox{ copies of } y_m)\]
of $\left(\R[G]_{(1,0)}\right)^l \oplus \left(\R[G]_{(0,1)}\right)^k$
yields the following definition
and proposition.

\begin{defi}
Let $G=(A\uplus B,E)$ be a bipartite graph and let
$\Theta\in \GL_n(\R)\times GL_m(\R)$ be a block-generic
matrix as in the previous section. Let $R^{(k,l)}(G)$  be an
$|E|\times (l|A|+k|B|)$ matrix whose rows are labeled by
the edges of $G$,  whose columns occur in blocks of size $l$
for each vertex in $A$ and blocks of size $k$ for each vertex
in $B$, and whose block corresponding to $v\in V$ and
$ab'\in E$ is given by
\[ \left\{
\begin{array}{ll}
(\theta_{i'b'}:\ 1\leq i\leq l) & \mbox{ if $v=a$},\\
(\theta_{ia}:\ 1\leq i\leq k) & \mbox{ if $v=b$}, \\
\, \, 0 &  \mbox{ if $v\notin \{a,b\}$}.
\end{array}
\right.
\]
The matrix $R^{(k,l)}(G)$ is called the
\emph{bipartite $(k,l)$-rigidity matrix of $G$}.
\end{defi}

\begin{prop} \label{rigid/stressfree-criterion}
Let $G=(A\uplus B, E)$ be a  bipartite graph. Then
$G$ is $(k,l)$-stress free if and only if
the rows of $R^{(k,l)}(G)$ are linearly independent, and
$G$ is $(k,l)$-rigid if and only if $\rank(R^{(k,l)}(G))=l|A|+k|B|-kl.$
\end{prop}

Several remarks are in order. For a matrix $M$, let $\row(M)$
denote the span of the rows of $M$.

\begin{rem} \label{rank(K_(n,m))}
Observe that for any bipartite graph $G=(A\uplus B, E)$,
$\row(R^{(k,l)}(G))\subseteq \row(R^{(k,l)}(K_{A,B}))$. On the other hand,
Proposition \ref{rigid/stressfree-criterion} implies that
$G$ is $(k,l)$-rigid if and only if
$\rank (R^{(k,l)}(G))=l|A|+k|B|-kl$. Since $K_{A,B}^b=K_{A,B}$, the graph
$K_{A,B}$ is $(k,l)$-rigid, and so
\[\rank(R^{(k,l)}(K_{A,B}))=l|A|+k|B|-kl.\]
Therefore, $G$ is $(k,l)$-rigid if and only if
$\row(R^{(k,l)}(G))=\row(R^{(k,l)}(K_{A,B}))$.

In particular, it follows that as for classical combinatorial
rigidity theory, there is a matroid underlying bipartite rigidity theory ---
namely the matroid represented by the rows of the $(k,l)$-rigidity matrix
of $K_{A,B}$. A $(k,l)$-rigid graph is one whose edges are a spanning set for
this matroid; a $(k,l)$-stress free graph is one whose edges are independent; and a
$(k,l)$-rigid and stress free graph is a basis.
\end{rem}

It is also worth remarking that in the case of $k=l$,
our rigidity matrix coincides with the completion matrix
defined in \cite[eq.~(4.2)]{SingerCucuringu}. As a result,
a bipartite graph is $(k,k)$-rigid in our sense if and only
if it is a rectangular graph that is locally completable
in dimension $k$ in the sense of \cite{SingerCucuringu}.
(Given a matrix some of whose entries are known, the associated
``rectangular'' graph is the bipartite graph whose vertices correspond
to rows and columns of the matrix and whose edges correspond to the
known matrix entries.)

In addition,  in the case of $k=l$, our
rigidity matrix is related to Kalai's hyperconnectivity
matrix \cite{Kalai-hyperconnectivity}, $H^{k,<'}(G)$,
computed w.r.t.~the entries of $\Theta$
and a total order $<'$ on the vertices. The rank of $H^{k,<'}(G)$
is independent of the total order chosen (this follows from
the fact that the result of exterior algebraic shifting is
independent of the labeling of the vertices
\cite{Kalai:fVectorsFamiliesConvexSetsI-84, Kalai-AlgShifting}).
Hence for a bipartite graph $G$, we let $<'$ be any order that places
vertices of $A$ before those of $B$. The rigidity matrix $R^{(k,k)}(G)$
is then simply the transpose of the matrix obtained
from $H^{k,<'}(G)$ by multiplying the $B$-labeled rows of
$H^{k,<'}(G)$ by $-1$. In particular,
$\rank(R^{(k,k)}(G))=\rank(H^{(k,<')}(G)).$
Therefore, we have:
\begin{rem}\label{rem:Hyperconnectivity}
A bipartite graph $G$ is $(k,k)$-stress free
if and only if it is $k$-acyclic in the sense
of \cite{Kalai-hyperconnectivity}.
\end{rem}

Finally, we notice that the notions of $(k,l)$-stress freeness
and $(k,l)$-rigidity introduced here are equivalent to the
geometric notions of Section 1.2.
It follows from Proposition \ref{rigid/stressfree-criterion}
that $G$ is $(k,l)$-stress free if and only if the left kernel
(i.e., the space of linear dependencies of rows) of the rigidity
matrix $R^{(k,l)}(G)$ equals $(0)$. Similarly, by
Remark \ref{rank(K_(n,m))}, $G$ is $(k,l)$-rigid if and only if
$\row(R^{(k,l)}(G))=\row(R^{(k,l)}(K_{A,B}))$, which happens if and
only if $\ker(R^{(k,l)}(G))=\ker(R^{(k,l)}(K_{A,B}))$. Thus, considering
a $(k,l)$-embedding of $G$ given by
$\phi(a)=(\theta_{ia} \ : \ 1\leq i\leq k)\times (0)$ for $a\in A$
and $\phi(b)=(0)\times (\theta_{i'b} \ : \ 1\leq i\leq l)$ for $b\in B$,
we obtain:

\begin{rem} \label{geom-int}
A bipartite graph $G=(A\uplus B, E)$ is $(k,l)$-stress free
if and only if for a {\em generic} $(k,l)$-embedding $\phi$,
the condition in eq.~(\ref{e:iii-bip}) holds.
A bipartite graph $G$ is $(k,l)$-rigid if and only if for a generic
$(k,l)$-embedding $\phi$, the condition in eq.~(\ref{e:iv-bip}) holds.
\end{rem}

We are now in a position to establish the cone,
deletion, contraction, and gluing lemmas, paralleling the
corresponding statements in classical rigidity.

\begin{lem}[Deletion Lemma]\label{lem:bipDeletion}
Let $G$ be a bipartite graph, $v$ a vertex of
$G$ of degree $d$, and $G'=G-v$ the graph obtained
from $G$ by deleting $v$.
\begin{enumerate}
\item[1.] If $G'$ is $(k,l)$-stress free and
$d\leq \left\{\begin{array}{cc} l & \mbox{ if $v\in A$}\\
                         k & \mbox{ if $v\in B$}
\end{array}
\right.$, then $G$ is $(k,l)$-stress free.
\item[2.] If $G'$ is $(k,l)$--rigid and
$d\geq \left\{\begin{array}{cc} l & \mbox{ if $v\in A$}\\
                         k & \mbox{ if $v\in B$}
\end{array}
\right.$,
then $G$ is $(k,l)$-rigid.
\end{enumerate}
\end{lem}
\begin{proof}
Assume $v\in A$ (the case $v\in B$ is very similar).
The matrix $R^{(k,l)}(G)$ is obtained from $R^{(k,l)}(G')$
by adjoining the $l$ columns corresponding to $v$ and the $d$ rows
corresponding to the edges containing $v$. As $v$ is not
the end-point of any edge of $G'$, these $l$ new columns
consist of zeros followed by a generic $d\times l$ block.
Thus,
\begin{equation} \label{rank+min}
\rank(R^{(k,l)}(G))\geq \rank(R^{(k,l)}(G'))+\min\{d,l\}.
\end{equation}

Now, if $G'$ is $(k,l)$-stress-free and $d\leq l$, then by
(\ref{rank+min}) and Proposition \ref{rigid/stressfree-criterion},
\[
|E(G)| \geq \rank(R^{(k,l)}(G))
\geq \rank(R^{(k,l)}(G'))+\min\{d,l\}=|E(G')|+d=|E(G)|.
\]
Hence $\rank(R^{(k,l)}(G))=|E(G)|$, and so $G$ is $(k,l)$-stress free
by Proposition \ref{rigid/stressfree-criterion}.

Similarly, if $G'$ is $(k,l)$-rigid and $d\geq l$, then by (\ref{rank+min})
and Proposition \ref{rigid/stressfree-criterion},
\begin{eqnarray*}
l|A|+k|B|-kl &\geq& \rank(R^{(k,l)}(G))
\geq \rank(R^{(k,l)}(G'))+\min\{d,l\}\\
&=& \left[ l(|A|-1)+k|B|-kl\right]+l=l|A|+k|B|-kl.
\end{eqnarray*}
Hence $\rank(R^{(k,l)}(G))=l|A|+k|B|-kl$, and so $G$ is $(k,l)$-rigid
by Proposition \ref{rigid/stressfree-criterion}.
\end{proof}

\begin{lem}[Contraction Lemma]\label{lem:bipVertexSpliting}
Let $G=(V,E)$ be a bipartite graph, $v$ and $w$ two vertices of
$G$ that belong to the same part, $C$ the set of common
neighbors of $v$ and $w$, and $G'=(V-\{v\}, E')$  the graph obtained
from $G$ by contracting $v$ with $w$.
\begin{enumerate}
\item[1.] If $G'$ is $(k,l)$-stress free and
 $|C|\leq \left\{\begin{array}{cc} l & \mbox{ if $v\in A$}\\
                         k & \mbox{ if $v\in B$}
\end{array}
\right.$, then $G$ is $(k,l)$-stress free.

\item[2.] If $G'$ is $(k,l)$-rigid and
 $|C|\geq \left\{\begin{array}{cc} l & \mbox{ if $v\in A$}\\
                         k & \mbox{ if $v\in B$}
\end{array}
\right.$, then $G$ is $(k,l)$-rigid.
\end{enumerate}
\end{lem}
\begin{proof}
For both parts assume that $v,w\in A$ (the case of $v,w\in B$
is analogous), and let $M$ be the matrix obtained from $R^{(k,l)}(G)$
by replacing each $\theta_{iv}$ with $\theta_{iw}$,  $1\leq i\leq k$.

\smallskip\noindent{\bf Part 1:}
By Proposition \ref{rigid/stressfree-criterion}, to complete the
proof we must show that if $R^{(k,l)}(G')$ has linearly independent rows,
then so does $R^{(k,l)}(G)$. As $M$ is a specialization of $R^{(k,l)}(G)$,
it suffices to check that the rows of $M$ are linearly independent.
And, indeed, if there is a linear dependence among the rows
of $M$, then it induces the same dependence (i.e., with the same
coefficients) among the rows of the matrix $M'$ obtained from $M$
by adding the columns of $v$ to the columns of $w$ and deleting
the columns of $v$. However, since $G'$ is obtained from $G$ by
contracting $v$ with $w$, it follows that the matrix $M'$ is obtained
from $R^{(k,l)}(G')$ by duplicating $|C|$ rows: for each $c\in C$,
the row of $R^{(k,l)}(G')$ labeled by $wc$ appears in $M'$ twice ---
once labeled by $wc$ and another time by $vc$.
As the rows of $R^{(k,l)}(G')$  are linearly independent, we conclude
that each nontrivial dependence among the rows of $M$
is supported on the rows labeled by $\{vc, wc  :  c\in C\}$.
Since $|C|\leq l$, and since the restriction of these
$2|C|$ rows to the $2l$ columns of $v$ and $w$ is of
the form {\tiny $\left[\begin{array}{cc} Z & 0\\ 0 &Z\end{array}\right]$},
where $Z$ is a generic $|C|\times l$ matrix,
we infer that the rows of $M$, and hence also of $R^{(k,l)}(G)$
are linearly independent. The assertion of Part 1 follows.

\smallskip\noindent{\bf Part 2:}
According to Proposition \ref{rigid/stressfree-criterion},
it suffices to show that if $\rank(R^{(k,l)}(G'))= l(|A|-1)+k|B|-kl$,
then $\rank(R^{(k,l)}(G))=l|A|+k|B|-kl$. Since
Remark \ref{rank(K_(n,m))} implies that $(k,l)$-rigidity can be destroyed,
but not created by deleting edges, we
assume that $v,w$ have exactly $l$ common neighbors, as the extra edges can be deleted.
Hence, $|E'|=|E|-l$.

The argument used in Part 1 leads to a stronger statement:
if $|C|=l$, then
\[\dim \Lker (R^{(k,l)}(G)) \leq
\dim \Lker (M) \leq \dim \Lker (R^{(k,l)}(G')).
\]
(Here $\Lker$ denotes the left kernel.) Indeed, if $|C|=l$,
then the matrix $Z$ from Part 1 is invertible, and so
the map sending a vector $(u_{e})_{e\in E}\in \Lker (M)$
to $(u'_{e})_{e\in E'}\in \Lker (R^{(k,l)}(G'))$,
where $u'_{wc}=u_{vc}+u_{wc}$ if $c\in C$ and
$u'_e=u_e$ otherwise, is injective. Therefore,
\begin{eqnarray*}
\rank(R^{(k,l)}(G)) &=&
|E|-\dim \Lker (R^{(k,l)}(G)) \\
  &\geq& |E'|+l -\dim \Lker (R^{(k,l)}(G')) \\
&=& \rank (R^{(k,l)}(G'))+ l  \\
 &=& l(|A|-1)+k|B|-kl+l \\
&=& l|A|+k|B|-kl,
\end{eqnarray*}
and the result follows.
\end{proof}

The following lemma is a bipartite analog of the gluing
lemma \cite[Theorem 2]{Asi-Roth2} (in the plane)
and \cite[Lemma 11.1.9]{Whiteley96-SomeMatroids}
(for the general case) that treats generic rigidity for the
union of general graphs.

\begin{lem}[Gluing Lemma]\label{lem:bipGluing}
Let $G=(A\uplus B, E)$ be a bipartite graph written as
the union $G=G_1\cup G_2$ of two bipartite graphs
$G_1=(A_1\uplus B_1, E_1)$ and $G_2=(A_2\uplus B_2, E_2)$.
\begin{enumerate}
\item[1.] If $G_1$ and $G_2$ are $(k,l)$-rigid, $|A_1\cap A_2|\geq k$,
and $|B_1\cap B_2|\geq l$, then $G$ is $(k,l)$-rigid.
\item[2.] If $G_1$ and $G_2$ are $(k,l)$-stress free,
and $G_1\cap G_2$  is $(k,l)$-rigid, then $G$ is
$(k,l)$-stress free.
\item[3.] If $G_1$ and $G_2$ are $(k,l)$-stress free,
and either $|A_1\cap A_2|\leq k$ and $|B_1\cap B_2|=0$, or
$|A_1\cap A_2|=0$ and $|B_1\cap B_2|\leq l$,
then $G$ is $(k,l)$-stress free.
\end{enumerate}
\end{lem}
\begin{proof}
To prove Part 1, by Remark \ref{rank(K_(n,m))} we may assume that
$G_1$ and $G_2$ are complete bipartite graphs. Construct $G$ from $G_1$ by
adding the vertices of $G_2\setminus G_1$ one by one;
when adding a vertex $v$ add also the edges in $G$ between $v$
and the former vertices (namely, the vertices of $G_1$ and
the vertices of $G_2\setminus G_1$ that were added before $v$).
Note that since $|A_1\cap A_2|\geq k$ and since each vertex
$v\in B_2$ is connected to all vertices of $A_1\cap A_2$, every time we
add a vertex $v\in B_2\setminus B_1$, we add it as a vertex of
degree at least $k$. Similarly, every time we add a vertex
$v\in A_2\setminus A_1$, we add it as a vertex of degree at least $l$.
Since $G_1$ is $(k,l)$-rigid, the Deletion Lemma
(Lemma \ref{lem:bipDeletion}) combined with induction
implies that all graphs in this sequence, including $G$, are $(k,l)$-rigid.

To prove Parts 2 and 3, consider the spaces
$\row(R^{(k,l)}(G_i))$, $\row(R^{(k,l)}(K_{A_i,B_i}))$ ($i=1,2$) as well as
$\row(R^{(k,l)}(G_1\cap G_2))$ and $\row(R^{(k,l)}(K_{A_1,B_1}\cap K_{A_2,B_2}))$
as subspaces of $\R^{l|A|+k|B|}$.
Note that under the conditions of either of the Parts 2 and 3
\[\row(R^{(k,l)}(G_1\cap G_2))=\row(R^{(k,l)}(K_{A_1,B_1}\cap K_{A_2,B_2})).\]
In the case of Part 2, this follows from the $(k,l)$-rigidity of $G_1\cap G_2$,
and in the case of Part 3, from the equality of graphs:
$K_{A_1,B_1}\cap K_{A_2,B_2}=G_1\cap G_2$
(indeed, both graphs are edgeless graphs on the same number of vertices).

Our proof relies on Lemma \ref{intersection} below.
As $\row(R^{(k,l)}(G_i))\subseteq \row(R^{(k,l)}(K_{A_i,B_i}))$ (for $i=1,2$),
Lemma \ref{intersection} yields
\begin{eqnarray*}
\row(R^{(k,l)}(G_1)) \cap \row(R^{(k,l)}(G_2)) &\subseteq&
\row(R^{(k,l)}(K_{A_1,B_1}))\cap\row(R^{(k,l)}(K_{A_2,B_2}))\\
&=&\row(R^{(k,l)}(K_{A_1\cap A_2,B_1\cap B_2}))\\
&=& \row(R^{(k,l)}(G_1\cap G_2)).
\end{eqnarray*}

Assume now that the rows $(R_e: e\in E)$ of $R^{(k,l)}(G)$
satisfy an $\R$-linear dependence
\begin{equation}\label{eq:dependence}
\sum_{e\in E_1}\alpha_e R_e = \sum_{e\in E_2\setminus E_1}\alpha_e R_e.
\end{equation}
Since the left-hand side of (\ref{eq:dependence}) is evidently
in $\row(R^{(k,l)}(G_1))$ and the right-hand side  is
in $\row(R^{(k,l)}(G_2))$, the previous inclusion implies
that the expression on the left-hand side of (\ref{eq:dependence})
is in the row span of $R^{(k,l)}(G_1\cap G_2)$.
Thus, the left-hand side of (\ref{eq:dependence})
can be rewritten using only edges $e\in E_1\cap E_2$,
and as $G_2$ is $(k,l)$-stress free,
all the coefficients on the right-hand side of (\ref{eq:dependence}) are
zeros. Then, as $G_1$ is $(k,l)$-stress free,
all the coefficients on the left-hand side of (\ref{eq:dependence}) are zeros
as well. Hence, $G$ is $(k,l)$-stress free.
\end{proof}

For $i\in[l], a\in A$, let $e_{i,a}$ denote the unit vector of
$\R^{l|A|+k|B|}$ with the coordinate $1$ in the $i$th of the $l$ slots
allotted for $a$ and zeros everywhere else. Define
$e_{j,b'}$ for $j\in[k],b'\in B$ similarly. Using this notation, the row of
$R^{(k,l)}(K_{A,B})$ corresponding to the edge $ab'$ can be written as
$\sum_{i=1}^l \theta_{i'b'}e_{i,a}+\sum_{j=1}^k \theta_{ja}e_{j,b'}$.

To finish the proof of the Gluing Lemma, it only remains to verify the
following.
\begin{lem} \label{intersection}
Let $A=A_1\cup A_2$ and $B=B_1\cup B_2$
be finite sets such that either
\begin{enumerate}
\item [(i)] $|A_1\cap A_2|\geq k$ and
$|B_1\cap B_2|\geq l$, %(called ``large'' intersection),
or
\item[(ii)]
$|A_1\cap A_2|\cdot |B_1\cap B_2|=0$, $|A_1\cap A_2|\leq k$, and
$|B_1\cap B_2|\leq l$. %(called ``small'' intersection).
\end{enumerate}
Let
$\V_1= \row(R^{(k,l)}(K_{A_1,B_1}))$,
$\V_2=\row(R^{(k,l)}(K_{A_2,B_2}))$,
$\V_\cap=\row(R^{(k,l)}(K_{A_1\cap A_2,B_1\cap B_2}))$
be three vector spaces considered as subspaces of $\R^{l|A|+k|B|}$.
Then $\V_1\cap \V_2=\V_\cap$.
\end{lem}
\begin{proof} We first treat case (i). %of a large intersection.
It suffices to show that $\V_1^{\perp}+\V_2^{\perp}=\V_\cap^{\perp}$,
where $\V^{\perp}$ denotes the orthogonal complement of $\V$
in $\R^{l|A|+k|B|}$ (equivalently, it denotes the kernel of the
corresponding matrix). We will do this by explicitly computing
$\V_1^{\perp}$, $\V_2^{\perp}$, and $\V_\cap^{\perp}$.

For $r\in[l]$ and $p\in[k]$, define $w^{rp}\in \R^{l|A|+k|B|}$ by
$w^{rp}=\sum_{a\in A} \theta_{pa}e_{r,a} -
\sum_{b'\in B} \theta_{r'b'}e_{p,b'}$, where $r'$ is the element
of $[l']$ corresponding to $r$ in $[l]$.
Note that $w^{rp}$ is orthogonal to all rows of $R^{(k,l)}(K_{A,B})$,
and hence also to all elements of $\V_1$. Thus
\[
\B_1:=\{w^{rp} :r\in[l], p\in[k]\} \cup
\{e_{i,a} : i\in[l], a\in A\setminus A_1\}\cup
\{e_{j,b'} : j\in[k], b'\in B\setminus B_1\} \subset \V_1^\perp.
\]
Moreover, the vectors of $\B_1$ are linearly independent:
indeed, using the unit vectors appearing in the above union,
we only need to check that the set
$\{w^{rp}_1:=\sum_{a\in A_1} \theta_{pa}e_{r,a} -
\sum_{b'\in B_1} \theta_{r'b'}e_{p,b'} :r\in[l], p\in[k]\}$
is linearly independent.
However, restricting the matrix formed by these $kl$ row vectors to
the columns of the first $k$ vertices in $A_1$ already yields an
invertible $kl\times kl$ matrix, as Gauss
elimination shows.

Since $|A_1\cap A_2|\geq k$ and $|B_1\cap B_2|\geq l$, the graph
$K_{A_1, B_1}$ is $(k,l)$-rigid. Hence
$\dim \V_1^{\perp}=kl+l(|A|-|A_1|)+k(|B|-|B_1|)= |\B_1|$, and we obtain
that $\B_1$ is a basis of $\V_1^{\perp}$. The same reasoning leads
to analogous expressions for bases $\B_2$ and $\B_\cap$ of $\V_2^{\perp}$
and $\V_\cap^{\perp}$, respectively. The result follows since
$\B_1\cup \B_2=\B_\cap$.

In case (ii), we must show that $\V_1\cap\V_2=(0)$.
As a warm-up, if $|A_1\cap A_2|=|B_1\cap B_2|=0$,
then the above description of $\B_1$ and $\B_2$ yields that
$\B_1\cup \B_2$ is a spanning set for $\R^{l|A|+k|B|}$,
and hence completes the proof.

If, say, $|A_1\cap A_2|\leq k$ and $B_1\cap B_2=\emptyset$,
then by definition of $R^{(k,l)}(G)$
\begin{equation} \label{eq:e}
\V_1\cap\V_2\subseteq \Span\{e_{i,a}: i\in[l],\ a\in A_1\cap A_2\}.
\end{equation}
However, since for a fixed $r\in [l]$, the $k$ scalar products
(where $p$ ranges over $[k]$)
$$\langle w^{rp},
\sum_{i\in[l]}\sum_{a\in A_1\cap A_2} \alpha_{ia} e_{i,a}\rangle
=\sum_{a\in A_1\cap A_2} \alpha_{ra}\theta_{pa}$$
vanish simultaneously only if $\alpha_{ra}=0$ for all $a$,
we infer that no nonzero vector from the right-hand side of
eq.~(\ref{eq:e}) is orthogonal to all $w^{rp}$. Thus
$\V_1\cap\V_2=(0)$, as required.
(The case of $|B_1\cap B_2|\leq l$ and $A_1\cap A_2=\emptyset$
is treated similarly.)
\end{proof}

We finish this section with the Cone lemma. This will require the
following definition.

\begin{defi}
Let $G=(A\uplus B,E)$ be a bipartite graph, where
$A=[n]$ and $B=[m']$.
Let $A^*:=A\cup\{0\}$ and $B^*:=B\cup\{0'\}$.
The \emph{left-side cone over $G$}, $C^{L}G$, is the
bipartite graph with the vertex set $A^*\uplus B$ and the edge set
$E\cup\{0b' : b'\in B\}$. The \emph{right-side cone over $G$},
$C^{R}G$, is the bipartite graph with the vertex set $A\uplus B^*$
and the edge set $E\cup\{a0' : a\in A\}$.
\end{defi}

To compute the balanced shifting of $C^{L}G$,
we extend our order $<$ on $V$ to an order  $<_0$ on
$A^*\cup B$ by requiring that $0$ is the smallest vertex.
Similarly, to work with  $C^{R}G$,
we extend $<$ to an order $<_{0'}$ on $A\cup B^*$ by requiring that
$0'$ is the smallest vertex. Note that if $<$ is $(k,l)$-admissible,
then $<_0$ is $(k+1,l)$-admissible and $<_{0'}$
is $(k,l+1)$-admissible.

\begin{lem}[Cone Lemma]\label{lem:bipCone}
The operations of coning and shifting commute, that is,
\[
(C^LG)^{b,<_0}= C^{L}(G^{b,<}) \quad \mbox{and}
\quad (C^{R}G)^{b,<_{0'}}= C^{R}(G^{b,<}).
\]
Thus, $G$ is $(k,l)$-rigid if and only if $C^{L}G$ is
$(k+1,l)$-rigid (equivalently, if and only if
$C^{R}G$ is $(k,l+1)$-rigid), and $G$ is $(k,l)$-stress free
if and only if $C^{L}G$ is $(k+1,l)$-stress free (equivalently,
if and only if $C^{R}G$ is  $(k,l+1)$-stress free).
\end{lem}
\begin{proof}
The proof that coning and shifting commute is very similar to that of the cone lemma
for the case of symmetric shifting, see
\cite[Lemma 3.3(4)]{Babson-Novik-Thomas-Cone}. We omit the details.
The second statement then follows from the combinatorial definitions of
``rigid'' and ``stress-free'' in Definition \ref{comb-rig}.
\end{proof}
%%%%%%%%%%%%%%%%%%%%%%%%
%%%%%%%%%%%%%%%%%%%%%%%%5

\section{Bipartite planar graphs}\label{sec:BipPlanarGraphs}
In this section we establish a bipartite analog of the rigidity
criterion for planar graphs.
Recall that according to Proposition \ref{prop:K_5}, for a graph $G$,
the existence of $K_5$ in $G^s$ is an obstruction to the planarity of $G$.
Here we show that for a bipartite $G$, the existence of $K_{3,3}$ in $G^{b,<}$
(where $<$ is a $(2,2)$-admissible order) is also an obstruction to the planarity
of $G$, that is, we prove the following more precise version of
Theorem \ref{thm:planar-1}:

\begin{theo}\label{theo:BipartiteGluck}
If $G$ is a  planar bipartite graph and $<$ is a
$(2,2)$-admissible order, then $K_{3,3}$ is not a subgraph of $G^{b,<}$.
Equivalently, planar bipartite graphs are $(2,2)$-stress free.
\end{theo}

Our proof of Theorem \ref{theo:BipartiteGluck} can be considered as
a bipartite analog of Whiteley's proof \cite{Wh-VertexSplitting}
of Gluck's result. It relies on the lemmas established in the previous
section as well as on some combinatorial properties of bipartite planar graphs.
The first such property is a bipartite
analog of the fact that any maximal planar graph with at
least $3$ vertices partitions the $2$-sphere into triangles;
the second is a bipartite analog of the the fact that
maximal planar graphs on $n$ vertices have $3n-6$ edges.
Both properties are well-known and included here only
for completeness.

A planar bipartite graph is \emph{maximal} if the
addition of any new edge (but no new vertices) results
in a graph that is either non-planar or non-bipartite.

\begin{lem}\label{lem:MaxPlanarBipartite}
If $G=(A\uplus B, E)$ is a maximal planar bipartite graph,
where $|A|,|B|\geq 2$, then  $G$ partitions the $2$-sphere into
$2$-cells whose boundaries are $4$-gons.
\end{lem}

\begin{proof}
Consider a planar drawing of $G$. If $G$ has a vertex of degree 0 or 1,
then $G$ is not maximal.
Thus we can assume that all degrees are at least $2$, and
hence that each edge is incident with two $2$-cells of the 2-sphere.
If one of the cells is not a $4$-gon, it has at least 6 vertices,
say, $(a,b,c,d,e,f, \ldots)$ in this order along its boundary.
By planarity of the drawing, not both $ad$ and $be$ are edges of $G$
(drawn outside of this $2$-cell). Since adding such a missing edge
and drawing it inside this $2$-cell preserves bipartiteness and planarity,
it follows that $G$ is not maximal.
\end{proof}

\begin{lem}\label{lem:2n-4}
If $G=(A\uplus B, E)$ is a maximal planar bipartite graph on $N$ vertices,
where $|A|,|B|\geq 2$, then $G$ has $2N-4$ edges.
Thus, if a planar bipartite $G$ with $N\geq 4$
vertices has  $2N-4$ edges, then $G$ is maximal.
\end{lem}

\begin{proof}
Let $e$ and $c$ be the number of edges and $2$-cells of $G$,
respectively. Each $2$-cell has $4$ edges and each edge is
contained in two $2$-cells. Thus $c=e/2$. By the Euler
formula $N+c-e=2$, and so $e=2N-4$.
\end{proof}

The following result will allow us to invoke the Contraction Lemma.

\begin{lem}\label{lem:bipartiteContraction}
Let $G$ be a maximal planar bipartite graph on $N$ vertices, $N>4$.
Then every $2$-cell induced by a planar drawing of $G$
has a pair of opposite vertices with exactly two
common neighbors, namely, the other two vertices on the
boundary of this $2$-cell. Assume $v,w$ form such a pair.
Then the graph $G'$ obtained from $G$ by contracting  $v$ with $w$
is a maximal planar bipartite graph on $N-1$ vertices.
\end{lem}

\begin{proof}
To prove the first assertion, note that if $(v,b,w,c)$
is a 4-cycle in a planar drawing of $G$ such that
 $v,w$ have another common neighbor $a$, and $b,c$
have another common neighbor $x$, then exactly one of
$a,x$ is inside the cycle and the other outside. In
particular, $(v,b,w,c)$ does not bound a 2-cell.

Let $v,w$ be a pair guaranteed by the first part.
Deleting $v$ from the drawing of $G$ creates
one new $2$-cell $X$, with boundary cycle
$(w,c,x_1,y_1,x_2,...,y_{k-1},x_k,d,w),$
where $c,d$ are the common neighbors of $v,w$ in $G$.
To obtain a planar drawing of $G'$, draw the edges
$wy_i$ (replacing the edges $vy_i$) inside the
cell $X$ according to this order. The graph $G'$ has one
vertex and two edges fewer than $G$
(indeed, the vertex $v$ and the edges $vc$ and $wd$ are ``lost''),
and so by Lemma \ref{lem:2n-4}, $G'$ is maximal.
\end{proof}

We are now in a position to prove Theorem \ref{theo:BipartiteGluck}.
\begin{proof}
Let $G=(V,E)$ be as in the theorem.
Since $G^b$ is bipartite with the same vertices as $G$ on each side,
it follows that if $G$ has a side with at most one vertex, then
$K_{3,3}\not\subseteq G^{b,<}$ for any order $<$. Thus
assume that $G$ has at least $2$ vertices on each side.

It is easy to see from the definition of balanced shifting that
if $H=(A\uplus B, E_H)$ is a subgraph of $G=(A\uplus B, E)$,
then $H^{b,<}$ is a subgraph of $G^{b,<}$.
(This follows from the fact that the Stanley-Reisner ideals of
$H$ and $G$ satisfy $I_H\supseteq I_G$.)
Hence, we assume without loss of generality that $G$ is
maximal. We prove by induction on the size of $V$ that such $G$ is
$(2,2)$-stress free (and also $(2,2)$-rigid). The base case, namely
$|V|=4$, does hold as in this case $G=G^b=C_4$. Thus assume $|V|>4$,
and consider vertices $v,w$ of $G$ as in
Lemma \ref{lem:bipartiteContraction}. Let $G'$ be the graph
obtained by contracting $v$ with $w$.
Lemma \ref{lem:bipartiteContraction}, the induction hypothesis,
and the Contraction Lemma (Lemma \ref{lem:bipVertexSpliting})
complete the proof.
\end{proof}

In view of Proposition \ref{prop:K_5},
a remaining natural problem is to find a notion of a minor for
bipartite graphs, denoted $<_b$, for which $K_{3,3}<_b G$
 would imply that $G$ is not planar, and $K_{3,3}\subseteq G^b$ would
imply that $K_{3,3}<_b G$. In a separate paper joint with
Chudnovsky and Seymour \cite{Chud.Kal.Nev.Nov.Sey.-BipMinors},
we propose such a notion of minors, $<_b$, and prove a bipartite analog of
Wagner's planarity criterion:
\emph{A bipartite graph $G$ is planar if and only if
$K_{3,3} \nless_b G$.}

A notion closely related to planarity of graphs is that of
\emph{linkless embeddability} of graphs in $\R^3$.
A graph $G$ is called {\em linklessly embeddable} if there
is an embedding of $G$ into $\R^3$ in such a way that every
two cycles of $G$ have zero linking number.
(A subfamily of linklessly embeddable graphs is that of {\em apex graphs}:
graphs that can be made planar by the removal of a single vertex.)
It is a theorem of Sachs \cite{Sachs83} that $K_{4,4}$ minus an edge,
which we denote by $K_{4,4}^-$,
is not linklessly embeddable. This fact and Theorem \ref{theo:BipartiteGluck}
lead us to the following conjecture.

\begin{conj}\label{conj:bipLinklessGraph}
Let $G=(A\uplus B, E)$ be a bipartite linklessly embeddable graph
and  $<$ a $(3,3)$-admissible order. If $|A|,|B|\geq 4$
then $K_{4,4}^-$ is not a subgraph of $G^b$, and thus
$E\leq 3|V(G)|-10$.
In particular, all bipartite linklessly embeddable graphs are $(3,3)$-stress free;
hence if $A$ and $B$ are each of size at least 3, then $E\leq 3|V(G)|-9$.
\end{conj}

At present, even the inequality $E\leq 3|V(G)|-9$ of the
above conjecture is open.
(Equality $|E|=3|V|-9$ holds for complete bipartite graphs $K_{3,m}$; these
graphs are apex graphs, and hence linklessly embeddable.)
As for the rest of the conjecture, the following weaker statement
is easy to prove.

\begin{lem}
Let $G$ be a linklessly embeddable bipartite graph. Then $G$ is
$(7,7)$-stress free, and so $G^{b,<}$ does not contain $K_{8,8}$,
where $<$ is any $(7,7)$-admissible order.
\end{lem}
\begin{proof}
Note that $G$ does not contain $K_6$ as a minor; this is an easy part of the
forbidden minor characterization of linklessly embeddable graphs
\cite{RST-linkless}. Thus, by a result of  Mader \cite{Mader-K6free},
if $G$ has $N$ vertices then $G$ has fewer than $4N$ edges. Hence
there is a vertex $v$ in $G$ whose degree is at most $7$. Now we use
the Deletion Lemma (Lemma \ref{lem:bipDeletion}) and induction to
conclude that $G-v$ is $(7,7)$-stress free, and hence so is $G$.
\end{proof}

Some other remarkable phenomena from graph rigidity
theory can be considered in the context of  bipartite rigidity.
First, recall that Gluck's proof of the generic rigidity of
maximal planar graphs is based on the theorems of Steinitz
and Dehn--Alexandrov. Steinitz's theorem asserts that every
maximal planar graph is the graph of some 3-dimensional simplicial polytope,
while the Dehn--Alexandrov theorem posits that the graph of any
$3$-dimensional simplicial polytope is infinitesimally rigid.
We do not know if the results of Dehn and Alexandrov have bipartite analogs.
Second, the non-generic embeddings of maximal planar triangulations for
which infinitesimal rigidity fails are also quite fascinating (e.g.,
in view of Bricard's Octahedra and Connelly's flexible
spheres \cite{Connelly-counterexample}). The analogous
situation for infinitesimal $(2,2)$-rigidity of bipartite planar
quadrangulations is also very interesting.

%%%%%%%%%%%%%%%%%%%%%%%
%%%%%%%%%%%%%%%%%%%%%%%
\section{Laman-type results for bipartite graphs}\label{sec:Laman}
In this section we apply the theory developed so far to bipartite trees
and outerplanar graphs.
We also consider a bipartite analog of the Laman theorem.
This theorem, see \cite{Laman}, provides a combinatorial characterization of
minimal (with respect to deletion of edges) generically $2$-rigid graphs:

\begin{theo}[Laman]\label{thm:classicLaman}
A graph $G=(V,E)$ is minimal generically $2$-rigid if and only if the
following conditions hold:
\begin{itemize}
\item[(i)] $|E|=2|V|-3$, and
\item[(ii)] every induced subgraph $G[V']=(V',E')$ with $|V'|\geq 2$
satisfies  $|E'|\leq 2|V'|-3$.
\end{itemize}
\end{theo}
Inspired by this result, we consider the relation between
$(k,l)$-rigidity and the following combinatorial condition, analogous to the
above Laman condition.
\begin{defi}\label{def:Laman}
A bipartite graph $G=(A\biguplus B,E)$ with $|A|\geq k$ and
$|B|\geq l$ is called \emph{$(k,l)$-Laman} if
\begin{itemize}
\item[(i)] $|E|=l|A|+k|B|-kl$, and
\item[(ii)] every induced subgraph $G[V']=(A'\biguplus B',E')$ of $G$
with $|A'|\geq k$ and $|B'|\geq l$ has at most $l|A'|+k|B'|-kl$ edges.
\end{itemize}
\end{defi}

We say that a $(k,l)$-rigid graph $G$ is
\emph{$(k,l)$-minimal} if the deletion of an arbitrary edge of $G$
results in a graph that is not $(k,l)$-rigid.
By Proposition \ref{rigid/stressfree-criterion}, $G$ is $(k,l)$-minimal if and only
if $G$ is $(k,l)$-rigid and stress free.
The following implication holds:

\begin{prop}\label{prop:Laman--necessary}
If a graph $G$ is $(k,l)$-minimal then $G$ is $(k,l)$-Laman.
\end{prop}
\begin{proof}
As $G$ is $(k,l)$-minimal, $\rank(R^{(k,l)}(G))=l|A|+k|B|-kl=|E|$,
where the second equality holds by the minimality of $G$.
Hence condition (i) of Definition \ref{def:Laman} is satisfied.
If condition (ii) is violated for some induced subgraph $G[V']$,
then the rows of $R^{(k,l)}(G[V'])$ are linearly dependent,
also when viewed as rows of $R^{(k,l)}(G)$, contradicting the
fact that $G$ is $(k,l)$-minimal, and hence, in particular,
$(k,l)$-stress free.
\end{proof}

What about the converse statement?
It follows from Whiteley's paper \cite{Whiteley-scenes} that the converse
does hold if one of $k,l$ is equal to $1$:

\begin{theo}\label{thm:(l,1)-Laman}
For any $k\geq 1$, if a graph $G$ is $(k,1)$-Laman
then $G$ is $(k,1)$-minimal.
\end{theo}

\begin{proof}
In \cite[Def.~4.2]{Whiteley-scenes}, Whiteley considers the following
specialization of $R^{(k,1)}(G)$: for all $b\in B$,
$\theta_{1'b}$ is specialized to 1, and for all $a\in A$, $\theta_{ka}$
is specialized to 1. He then proves \cite[Thm.~4.1]{Whiteley-scenes} that
if $G$ is $(k,1)$-Laman then the rank of the resulting matrix is $|E|$,
and hence so is the rank of $R^{(k,1)}(G)$.
\end{proof}

On the other hand, as the following example shows, the converse of
Proposition \ref{prop:Laman--necessary} is false when both $k,l\geq 2$.

\begin{exa}\label{ex:(l,k)-LamanNOTrigid} \qquad
\begin{enumerate}
\item[1.] Let $G_1$ and $G_2$ be two copies of $K_{3,3}$
minus an edge, and let $G$ be obtained by gluing $G_1$ and $G_2$ along
the two vertices of the missing edge, denoted $ab$. Then
$G=(A\biguplus B,E)$ is $(2,2)$-Laman, but $G$
is not $(2,2)$-stress free, and hence is not $(2,2)$-minimal.
\item[2.] For $k,l\geq 2$, let
\[H=\underbrace{C^R\ldots C^R}_{l-2}\underbrace{C^L\ldots C^L}_{k-2} G\]
 be obtained from $G$ by iterative coning. Then $H$ is $(k,l)$-Laman,
but it is not $(k,l)$-minimal.
\end{enumerate}
\end{exa}
\begin{proof}
As $G_1^b$ and $G_2^b$ are both isomorphic to
$K_{3,3}$ minus the edge between $3$ and $3'$, it follows that
each of $G_1,G_2$ is $(2,2)$-rigid.
Hence in the rigidity matrix $R^{(2,2)}(G)$,
there is a non-trivial linear combination
of the rows of $G_1$ that equals the row of the missing edge
$ab$, and similarly for $G_2$; the difference of these two
linear combinations provides a non-zero linear dependence of the rows of
$R^{(k,l)}(G)$. Thus $G$ is not $(2,2)$-stress free, and hence it is
is not $(2,2)$-minimal. On the other hand, one readily checks that $G$ is
$(2,2)$-Laman. This completes the proof of Part 1.

Part 2 is an immediate consequence of Part 1.
Indeed, the Cone Lemma (Lemma \ref{lem:bipCone})
and the fact that $G$ is not $(2,2)$-minimal yield that $H$ is not $(k,l)$-minimal.
Further, it is straightforward to check that the left cone over an $(r,s)$-Laman graph
is $(r+1,s)$-Laman while the right cone over an  $(r,s)$-Laman graph
is $(r,s+1)$-Laman. As $G$ is $(2,2)$-Laman, we then conclude that $H$ is $(k,l)$-Laman.
\end{proof}
%%%%%%%%%%%%%%%%%%%%%%%\

It would be interesting to have a complete combinatorial
characterization of minimal $(k, l)$-rigid graphs (i.e., bases of
the $(k, l)$-rigidity matroid) even for $k = l = 2$.

We now consider the effect of balanced shifting on trees (which are bipartite)
and bipartite outerplanar graphs.

\begin{theo}  \label{forest-planar}
Let $G$ be a bipartite graph with sides $A$ and $B$,
and let $<$ be the total order on $A\cup B$ with respect to which
$G^b$ is computed.
\begin{enumerate}
\item If $<$ is $(1,1)$-admissible and $G$ is a forest then
$K_{2,2}\nsubseteq G^b$; equivalently, $G$ is $(1,1)$-stress free.
\item If $<$ is $(2,1)$-admissible and $G$ is outerplanar,
then $K_{3,2}\nsubseteq G^b$; equivalently, $G$ is $(2,1)$-stress free.
(Similarly, if $<$ is $(1,2)$-admissible and $G$ is outerplanar, then
$G$ is $(1,2)$-stress free.)
\item  If $<$ is $(2,2)$-admissible and $G$ is planar, then
$K_{3,3} \nsubseteq G^b$; equivalently, $G$ is $(2,2)$-stress free.
\end{enumerate}
\end{theo}
\begin{proof}
Part (1) is proved by induction: for the inductive step,
pick a leaf and use the Deletion Lemma (Lemma \ref{lem:bipDeletion}).
Part (3) is Theorem \ref{theo:BipartiteGluck}.
In Part (2), the order $<$ has exactly two vertices in $A$
among the least three vertices. Add to $B$
a new and smallest vertex, and connect it to all vertices of $A$.
This creates a bipartite planar graph and a $(2,2)$-admissible order $<_{0'}$.
Now the Cone Lemma (Lemma \ref{lem:bipCone}) and Part (3) complete the proof.
\end{proof}

\noindent We remark that Parts (1) and (2) of Theorem \ref{forest-planar}
also follow easily by counting the edges of induced subgraphs
and using Whiteley's criterion, Theorem \ref{thm:(l,1)-Laman}.

%%%%%%%%%%%%%%%%%%%%%%%
%%%%%%%%%%%%%%%%%%%%%%%
\section{Graphs of polytopes}\label{sec:Cubical}
\subsection{Cubical polytopes}
We now discuss potential applications of bipartite rigidity,
\'a la Kalai \cite{Kalai-LBT},
to lower bound conjectures on face numbers of cell complexes with
a bipartite $1$-skeleton.

Recall that by a result of Blind and Blind \cite{Blind-Blind:Gaps},
if $P$ is a cubical $d$-polytope with $d>2$,
then the graph $G(P)$ of $P$ is bipartite.
Moreover, if $d>2$ is even, then the two sides of $G(P)$
have the same number of vertices. (These results were
generalized to arbitrary cubical spheres by
Babson and Chan \cite{Babson-Chan}.) We are interested
in the cubical conjecture of Jockusch \cite{Jock},
see Conjecture \ref{conj:Adin}, asserting that
if $K$ is a cubical polytope of dimension $d\geq 3$ with $f_0(K)$
vertices and $f_1(K)$ edges, then

Note that if $G$ is $(2, d-1)$-rigid, and has the same number
of vertices on each side, then $G$ has at least
$\frac{d+1}{2}f_0(G)-2(d-1)$ edges. The graph $G(P)$ of
a stacked cubical polytope $P$ is bipartite and has the same number
of vertices on each side, but has only $\frac{d+1}{2}f_0(P)-2^{d-1}$ edges.
We will show in Proposition \ref{prop:stacked cubes rigidity}
 that for such $P$, it is possible to add to $G(P)$
exactly $2^{d-1}-2(d-1)$ edges, all in one facet of $P$, in such a way that
the resulting graph is $(2,d-1)$-rigid and stress free.
We will also establish a similar statement with respect to
$(1,d)$-rigidity and stress freeness, see
Proposition \ref{prop:stacked cubes (1,d)-rigidity}.

This yields the following approach to Jockusch's conjecture;
specifically, a positive answer to the following problem will
imply Conjecture \ref{conj:Adin} for all  \emph{even} $d>2$:
\begin{prob}\label{prob:Adin(2,d-1)}
Let $G$ be the graph of a cubical $d$-polytope, where $d>2$
is even. Is it possible to add $2^{d-1}-2(d-1)$ edges to $G$
to obtain a $(2,d-1)$-rigid graph? Is it possible to add
$2^{d-1}-d$ edges to $G$ to obtain a $(1,d)$-rigid graph?
\end{prob}

A similar reasoning shows that a positive answer to the following
problem will imply Conjecture \ref{conj:Adin} for an arbitrary $d$:
\begin{prob}\label{prob:Adin((d+1)/2,(d+1)/2)}
Let $G$ be the graph of a cubical $d$-polytope, where $d>2$.
Is it possible to add
$2^{d-1}-\lfloor\frac{d+1}{2}\rfloor \lceil \frac{d+1}{2}\rceil$
edges to $G$ to obtain a
$(\lfloor \frac{d+1}{2}\rfloor, \lceil\frac{d+1}{2})\rceil$-rigid
graph?
\end{prob}

We are now in a position to show how to add edges to the graph
of a {\em stacked cubical polytope} to make it $(2,d-1)$-rigid and stress-free.
(Recall that a stacked cubical polytope is a polytope obtained
starting with a cube and repeatedly gluing cubes onto facets.)
Our construction relies on the following lemmas.

\begin{lem}\label{lem:cube-edges in opposite facets}
For $d\geq 3$, let $C^d$ be the $d$-cube, and let $A$ and $B$ be the two
sides of the bipartite graph $G(C^d)$ of $C^d$.
Fix vertices $v,v^*\in A$ that are contained in a $2$-face of $C^d$.
Let $F$ and $F^*$ be opposite facets of $C^d$
such that $v\in F\cap A$ and $v^*\in F^*\cap A$ (they exist when
$d\geq 3$). Add to $G(C^d)$ all the edges $vb$ and $v^*b^*$ where
$b\in F\cap B$ and
$b^*\in F^*\cap B$ to obtain a new graph $G'(C^d)$.
Then $G'(C^d)$ is $(2,d-1)$-rigid and stress free.
\end{lem}
\begin{proof}
The proof is by induction on $d$. In the case of $d=3$ no edges are
added and $(2,2)$-stress freeness follows from
Theorem \ref{theo:BipartiteGluck} (or check directly).
Moreover, since the graph of the $3$-cube is a {\em maximal} planar
bipartite graph, it is also $(2,2)$-rigid.

Assume $d>3$. Then there exists a facet $H$ of $C^d$ containing both
$v$ and $v^*$; we let $H^*$ denote the opposite facet.
We now show, in four steps, that $(2,d-1)$-rigidity and
stress freeness of $G'(C^d)$ follow from $(2,d-2)$-rigidity and
stress freeness of $G'(H)\cong G'(C^{d-1})$ --- the graph formed
from the graph of $H$ in the same manner as $G'(C^d)$ is formed
from the graph of $C^d$.

\smallskip\noindent{\bf Step 1:} linearly order all
$t\in A\cap F\cap H^*$ and contract successively $t$ with $v$
(in $G'(C^d)$); similarly, for $t^*\in A\cap F^*\cap H^*$
contract successively $t^*$ with $v^*$; call the resulting graph $G_1$.
By construction of $G'(C^d)$, $v$ is connected to {\em every}
vertex in $F\cap B$ but only to one vertex in $F^*\cap B$
(namely, the common neighbor of $v,v^*$ in $F^*\cap B$) --- the vertex
that has no neighbors in $F\cap A$ except $v$. On the other hand, it is evident
from the structure of $G(C^d)$ that
$t$ has exactly $d-1$ neighbors in $F\cap B$. Therefore, it follows that $t,v$
have exactly $d-1$ common neighbors in $G'(C^d)$.
(The same argument also applies to $t^*,v^*$.)
Hence, by the Contraction Lemma (Lemma \ref{lem:bipVertexSpliting}),
to complete the proof,
it is enough to show that $G_1$ is $(2,d-1)$-rigid and stress free.

\smallskip\noindent{\bf Step 2:} successively contract pairs of vertices in
$B\cap F\cap H^*$ until a single vertex $p$ remains, and similarly contract
vertices in $B\cap F^*\cap H^*$ until a single vertex $p^*$ remains; call the
resulting graph $G_2$. At each contraction, the two identified vertices
have exactly two common neighbors, namely, $v$ and $v^*$.
Thus, by Lemma \ref{lem:bipVertexSpliting}, it suffices to prove
that $G_2$ is $(2,d-1)$-rigid and stress free.

\smallskip\noindent{\bf Step 3:} contract $p^*$ with $p$ to obtain $G_3$.
As $p$ and $p^*$ have two common neighbors in $G_2$, namely $v$ and $v^*$,
by Lemma \ref{lem:bipVertexSpliting}
it remains to verify the assertion for $G_3$.

\smallskip\noindent{\bf Step 4:} in $G_3$, $p$ is a right-cone vertex:
indeed, it is connected to all vertices that belong to side $A$.
Delete $p$ and all edges incident with it to obtain $G_4$.
By the Cone Lemma (Lemma \ref{lem:bipCone}) we need to show
that $G_4$ is $(2,d-2)$-rigid and stress free.

It remains to notice that $G_4$ is obtained from the
graph $G(H)$ of $H$ by adding all edges $vb$ where $b$ is a vertex
in the facet $H\cap F$ of $H$, and all edges $v^*b^*$
where $b^*$ is a vertex in the opposite facet $H\cap F^*$
of $H$. Thus, the assertion follows by induction.
\end{proof}

\begin{lem}\label{lem:cube-edges in one facet}
For $d\geq 3$, let $C^d$ be the $d$-cube, let
$A$ and $B$ be the two sides of the bipartite graph $G(C^d)$ of $C^d$,
and let $F$ be a facet of $C^d$. Fix two vertices $u,w\in F\cap A$, and
add to $G(C^d)$ all the edges $ub$ and
$wb$ where $b\in F\cap B$ to obtain a new graph $G$.
Then $G$ is $(2,d-1)$-rigid and stress free.
\end{lem}

\begin{proof}
If $d=3$, then no edges are added and $(2,2)$-rigidity and stress freeness
follow from Theorem \ref{theo:BipartiteGluck}. Thus, assume $d>3$.

As before, let $F^*$ be the facet of $C^d$ opposite to $F$, and let $H$ and $H^*$
be two opposite facets of $C^d$ such that $u\in H$ and $w\in H^*$. For a vertex
$b\in B\cap F^*$ denote by $b_F$ the unique neighbor of $b$ in $A\cap F$.

\smallskip\noindent{\bf Step 1:}
for every $b\in B\cap F^* \cap H$ contract $b_F$ with $u$, and
for every $b\in B\cap F^* \cap H^*$ contract $b_F$ with $w$;
call the resulting graph $G_1$. At each contraction,
the two identified vertices have $d-1$ common neighbors. Thus,
by the Contraction Lemma (Lemma \ref{lem:bipVertexSpliting})
it is enough to show that $G_1$ is $(2,d-1)$-rigid and stress free.

\smallskip\noindent{\bf Step 2:}
fix $p\in B\cap F$ and successively contract all other vertices in
$B\cap F$ with $p$ to obtain $G_2$. At each contraction, the two
identified  vertices have $u$ and $w$ as their only common
neighbors. Hence by Lemma \ref{lem:bipVertexSpliting}
it suffices to check that $G_2$ is $(2,d-1)$-rigid
and stress free.

\smallskip\noindent{\bf Step 3:}
observe that $p$ is a  right-cone vertex in $G_2$;
delete it to obtain $G_3$. By the Cone Lemma (Lemma \ref{lem:bipCone}),
the result will follow if we show that $G_3$ is $(2,d-2)$-rigid
and stress free.

\smallskip\noindent{\bf Step 4:}
fix two vertices $v\in A\cap F^*\cap H$ and $v^*\in A\cap F^*\cap H^*$ that
are contained in a $2$-face: such $v,v^*$ exist as $d>3$.
Contract $u$ with $v$ and $w$ with $v^*$ to obtain $G_4$. Since there are $d-2$
common neighbors at each contraction, by Lemma~\ref{lem:bipVertexSpliting},
it is enough to show that $G_4$ is $(2,d-2)$-rigid and
stress free. This, however, is an immediate consequence of
Lemma \ref{lem:cube-edges in opposite facets},
as, using the notation of that lemma, $G_4=G'(F^*)$.
\end{proof}

\begin{prop}\label{prop:stacked cubes rigidity}
For $d\geq 3$, let $P$ be a stacked cubical $d$-polytope,
let $A,B$ be the two sides of the bipartite graph of $P$,
and let $F$ be a facet of $P$.
Fix vertices $u,w\in F\cap A$ and add to the graph of $P$ all edges
$ub$ and $wb$ where $b\in F\cap B$ to obtain a new graph $G$.
Then $G$ is $(2,d-1)$-rigid and stress free.
\end{prop}

\begin{proof}
First, observe that $P$ can be formed by successively stacking cubes in a
certain order $C_1,C_2,...,C_m$ satisfying the condition that $F$ is a facet of
$C_1$: indeed, the graph whose vertices are the cubes $C_i$ and whose edges are
between the cubes that are glued along a facet, is a tree, and
so any cube can be taken to be the first in the stacking process.
Let $P_i$ be the stacked cubical sphere obtained by stacking $C_1,...,C_i$,
and let $G_i$ be the corresponding graph (with the added edges in $F$). In
particular, $G=G_m$. We show by induction that $G_i$ is
$(2,d-1)$-rigid and stress free.

For $G_1$, this follows from Lemma \ref{lem:cube-edges in one facet}, and so
assume $i>1$. By induction, $G_{i-1}$ is $(2,d-1)$-rigid, and hence
its rigidity matrix has the same rank as the $(2,d-1)$-rigidity matrix of
the complete bipartite graph on the same vertex set (with same sides as in
$G_{i-1}$), denoted $K(i-1)$. Thus the $(2,d-1)$-rigidity matrices of $G_i$ and
$G_i\cup K(i-1)$ have the same rank. Let $G'_i$ be the restriction of
$G_i\cup K(i-1)$ to the vertices of $C_i$. Then $G'_i$ is the graph
of the $d$-cube with all bipartite edges in one of its facets added.
By Lemma \ref{lem:cube-edges in one facet}, $G'_i$ is $(2,d-1)$-rigid.
Therefore, by the Gluing Lemma (Lemma \ref{lem:bipGluing}),
the union $G_i\cup K(i-1)=K(i-1)\cup
G'_i$ is $(2,d-1)$-rigid, and hence so is $G_i$.

Counting the number of edges in the graph $G_i$ with sides
$A_i\subseteq A$ and $B_i\subseteq B$ yields
\[
|E(G_i)|=(d+1)\cdot i \cdot 2^{d-1}-2(d-1)
=(d-1)|A_i|+2|B_i|-2(d-1).
\]
Thus, $G_i$ is also $(2,d-1)$-stress free.
\end{proof}

\begin{cor}\label{cor:stacked cubes stress free}
For $d\geq 3$, the graph of a stacked cubical $d$-polytope
is $(2,d-1)$-stress free. %$\square$
\end{cor}

Using the numerical condition of Theorem \ref{thm:(l,1)-Laman}
on $(1,d)$-minimality, we also establish the following
$(1,d)$-analog of Proposition \ref{prop:stacked cubes rigidity}.

\begin{prop} \label{prop:stacked cubes (1,d)-rigidity}
For $d\geq 4$, let $P$ be a stacked cubical $d$-polytope and $F$ a facet of
$P$. Then it is possible to add to the graph of $P$ exactly $2^{d-1}-d$ edges,
all of them in $F$, so that the resulting graph is $(1,d)$-Laman, and
hence $(1,d)$-rigid and stress free.
\end{prop}

The same proof as in Proposition \ref{prop:stacked cubes rigidity} shows that
the following two lemmas imply
Proposition \ref{prop:stacked cubes (1,d)-rigidity}.

\begin{lem} \label{lem:(d-1)-cube}
For $d\geq 4$, it is possible to add  $2^{d-1}-d$ edges to the graph of a
 $(d-1)$-cube so that the resulting graph is $(1,d)$-Laman.
\end{lem}

\begin{proof}
The proof is by induction on $d$. For $d=4$, we need to add  $2^3-4=4$
edges to a $3$-cube. Adding all long diagonals (there are exactly 4 of them)
results in $K_{4,4}$, which is easily checked to be $(1,4)$-Laman.

For the inductive step, consider two opposite facets $F'$ and $F''$
of a $(d-1)$-cube $P$, and assume that we can add  $2^{d-2}-(d-1)$ edges
to the graph $G(F')$ of $F'$ so that the resulting
graph is $(1,d-1)$-Laman, and the same amount
of edges to the graph $G(F'')$ of $F''$ so that the resulting graph
 is $(1,d-1)$-Laman. Also add
$d-2$ arbitrary ``bipartite'' edges that go between $F'$ and $F''$.
Thus, the total number of added edges is $2(2^{d-2}-(d-1))+(d-2)=2^{d-1}-d$.
We claim that the graph of $P$ with all the added edges is $(1,d)$-Laman.
And indeed, for any subgraph $G=((A'\cup A'')\uplus( B'\cup B''),  E)$ of this graph,
where $A'\cup B'$ is a subset of the vertex set of $F'$ and  $A''\cup B''$
of $F''$, and where we denote by $E(C,D)$ the set of edges connecting $C$ to $D$,
we have
\begin{eqnarray*}
 |E|&=& |E(A',B')|+|E(A'',B'')|+ |E(A',B'')|+|E(A'', B')|\\
 &\leq& [(d-1)|A'|+|B'|-(d-1)]+[(d-1)|A''|+|B''|-(d-1)] \\
 && \qquad + [|A'|+|A''|]+(d-2)\\
&=& d|A'\cup A''|+|B'\cup B''| -d.
\end{eqnarray*}
Some explanation is in order:
 in the second step, the first two summands follow from the inductive
hypothesis, the third summand, $|A'|+|A''|$, is implied by the fact that in the
original graph of $P$, for each vertex of $A'$ there is at most one
edge from this vertex to $B''$, and, similarly, for each vertex of
$A''$ there is at most one edge to $B'$; finally,
the $d-2$ added edges between the two facets contribute the last summand.
Further, in the above inequality,  equality holds when considering the entire
graph, and so this graph is $(1,d)$-Laman.
% The statement now follows from theorem \ref{thm:(l,1)-Laman}.
\end{proof}

\begin{lem} \label{lem:d-cube}
For $d\geq 4$, it is possible to add  $2^{d-1}-d$ edges to the graph of a
 $d$-cube, all of them in one facet, so that the resulting graph is
$(1,d)$-Laman.
\end{lem}

\begin{proof}
Let $P$ be a $d$-cube, and let $F'$ and $F''$ be a pair of opposite facets of
$P$. %on vertex sets $A'\uplus B'$ and $A''\uplus B''$, resp.
Using Lemma \ref{lem:(d-1)-cube}, add $2^{d-1}-d$ edges to the graph of
$F'$ to make it $(1,d)$-Laman. We claim that
the graph of $P$ together with these added edges is $(1,d)$-Laman.
And indeed, for any subgraph
$G=((A'\cup A'')\uplus(B'\cup B''), E)$ of the resulting graph,
\[
|E| \leq |E(A',B')| + d|A''| + |B''|  \leq d|A'\cup A''|+|B'\cup B''| -d,
\]
where in the first inequality the summand $d|A''|$ is justified by the fact
that each vertex of $A''$ has degree $d$, and hence cannot contribute more
than $d$ edges, while the summand $|B''|$ is explained by the fact that there
is at most one edge from each vertex of $B''$ to $A'$; the second inequality is by
Lemma \ref{lem:(d-1)-cube}. Further, in the above inequality, equality holds
when considering the entire graph, and hence this graph is $(1,d)$-Laman.
\end{proof}

\subsection{Dual to balanced polytopes}
For relevant terminology on simplicial complexes used below,
see Section \ref{sec:PrelimComplexes}.

Recall that the facet-ridge graph of a pure simplicial complex $K$ is the
graph whose vertices are facets of $K$, and two facets are connected
by an edge if they share a common ridge. Recall also that a {\em combinatorial
manifold} (without boundary) of dimension $d-1$
is a simplicial complex whose geometric realization is a $(d-1)$-manifold
with an additional restriction that all vertex links are piecewise linear homeomorphic
to the boundary of a $(d-1)$-simplex.

Let $K$ be a $(d-1)$-dimensional simplicial complex; assume further that
$K$ is a combinatorial manifold with a trivial
fundamental group.  According to Joswig \cite{Joswig2002},
the facet-ridge graph of $K$ is bipartite if and only if $K$ is balanced.
(For $d=3$ this is a classic result by Ore;
for $d=4$ this result goes back to Goodman and
Onishi \cite{GoodmanOnishi},
and to the unpublished work of Deligne, Edwards, MacPherson, and  Morgan.)
In particular, if $P$ is a balanced simplicial polytope and $P^*$
is a polytope dual to $P$, then the graph of $P^*$ is bipartite.
This graph is also $d$-regular, and hence has $f_{d-1}(K)$ vertices and
$d f_{d-1}(K)/2$ edges.

\begin{prob}\label{prob:JoswigRelated}
Fix $d\geq 3$.
Let $K$ be a $(d-1)$-dimensional balanced simplicial complex
and assume that $K$ is a combinatorial manifold
(without boundary) with a trivial fundamental group. Let $G$ be
the facet-ridge graph of $K$.
\begin{enumerate}
\item Is this graph $(\lfloor (d+1)/2\rfloor, \lceil (d+1)/2 \rceil)$-stress free?
\item
Assume further that $K$ has no missing facets.
Is $G$ a $(k, d-k)$-rigid graph for $1\leq k\leq d-1$?
\end{enumerate}
\end{prob}

We start with Part (1).
As the only 2-dimensional manifold with a trivial fundamental group is a sphere,
and as the facet-ridge graph of a 2-dimensional simplicial sphere is planar,
it follows from Theorem \ref{theo:BipartiteGluck} that the answer to
Problem \ref{prob:JoswigRelated}(1) is positive in the case of $d=3$.
Also it is well-known and easy to prove by induction on dimension
(by considering vertex links) that the number of facets of a
balanced $(d-1)$-dimensional manifold (without boundary) is at
least $2^d$, for all $d$.
Since $2^{d-1}\geq \lfloor (d+1)/2\rfloor \lceil (d+1)/2 \rceil$,
at least the inequality on the number of edges of $G$ implied
by Problem \ref{prob:JoswigRelated}(1) does hold for all values of $d$.

Next we discuss Part (2).
First, to see that the condition in Part (2) is necessary,
let $d\geq 4$ and consider $2d$ copies of the boundary complex of the
$d$-dimensional cross polytope $C^*$ (with a natural $d$-coloring).
We label these copies by $\partial C^*_0,
\partial C^*_1,\ldots, \partial C^*_{2d-1}$.
As the facet-ridge graph of $\partial C^*_i$ is bipartite,
we can refer to facets of $\partial C^*_i$ as belonging to
either side $A$ or side $B$ of this graph.
Pick $2d-1$ facets $H'_1,\ldots, H'_{2d-1}$ of $\partial C^*_0$
that belong to side $B$ (this is possible as there are $2^{d-1}$
such facets in total), and for each $i=1,\ldots, 2d-1$, pick a facet
$H_i$ of $\partial C^*_i$ that belongs to side $A$. Now,
for each $i=1,\ldots, 2d-1$ glue the complex $\partial C^*_i$
onto $\partial C^*_0$ by identifying each vertex of $H_i$ with
the same color vertex of $H'_i$, and removing
the resulting common facet. Denote the complex obtained in this
way by $K$. Thus $K$ is balanced. In fact, $K$ is isomorphic
to the boundary complex of a balanced simplicial $d$-polytope.

Let $G$ be the facet-ridge graph of $K$. Observe that for each $0\leq i \leq 2d-1$,
$\partial C^*_i$ has $2^d$ facets for the total number of $2d\cdot 2^d$ facets.
Since each gluing described above reduces the total number of facets by $2$, the complex
$K$ has $2d\cdot 2^d-2(2d-1)$ facets.
Hence $G$ has $d\cdot2^d-(2d-1)$ vertices on each side.
By $d$-regularity of $G$, we conclude that $G$ has $d^2\cdot 2^d -2d^2+d$ edges.
Note also that according to Proposition \ref{rigid/stressfree-criterion},
a $(1,d-1)$-minimal graph on the same vertex set as $G$ has
$d(d\cdot2^d-(2d-1))-(d-1)=d^2\cdot 2^d-2d^2+1$ edges.

 We claim that $G$ is not
$(1,d-1)$-rigid. Indeed, if $G$ were $(1, d-1)$-rigid  there would be a way to
delete $d-1$ edges of $G$ (corresponding to the ridges of $K$)
so that the resulting graph $G'$ is $(1, d-1)$-minimal,
and hence by Proposition \ref{prop:Laman--necessary}, $(1, d-1)$-Laman.
However, since each ridge belongs to only two facets, these deletions
affect at most $2(d-1)$ of our cross polytopes; in other words, for some
$1\leq i_0\leq 2d-1$, no ridges of $\partial C^*_{i_0}-\{H_{i_0}\}$ are deleted.
The subgraph of $G'$ induced by the facets of
$\partial C^*_{i_0}-\{H_{i_0}\}=A'\cup B'$ violates the $(1, d-1)$-Laman
condition: since $|A'|=2^{d-1}-1$ and $|B'|=2^{d-1}$,
the number of edges in this subgraph is
$d 2^{d-1}-d$ while $(d-1)|A'|+|B'|-(d-1)=d 2^{d-1}-2(d-1)$.
A similar construction works for every $k=1,\ldots, d-1$,
as well as for $d=3$ (where more copies of $C^*$ are glued together).

Note that if a bipartite graph $G=(A\cup B, E)$ is $(k,d-k)$-rigid then
it has at least $(d-k)|A|+k|B|-k(d-k)$ edges --- a quantity that is smaller than
the number of edges a $d$-regular bipartite graph $G$ has. %$d|V|$.
Thus at least the inequality on the number of edges of $G$ implied
by Problem \ref{prob:JoswigRelated}(2) does hold for all $1\leq k<d$.

We now give a positive answer to Problem \ref{prob:JoswigRelated}(2)
for the case of $d=3$.

\begin{prop}\label{prop:Joswig(1,d-1)}
Let $K$ be a a balanced $2$-dimensional simplicial sphere without missing triangles,
and let $G$ be the facet-ridge graph of $K$. Then $G$ is $(1,2)$-rigid.
\end{prop}
\begin{proof}
By Theorem \ref{thm:(l,1)-Laman}, it suffices to show that
$G$ has a subgraph $G''=G-\{ab,a'b'\}$ with the following property:
for every proper subset $V'=A'\uplus B'\subsetneq V(G)$, with both
$A'\subseteq A$ and $B'\subseteq B$
nonempty, the induced subgraph $G''[V']$ has at most $2|A'|+|B'|-2$ edges.
Let $e'=|E(G[V'])|$ and $e''=|E(G''[V'])|$. Thus, $e''\leq e'$, and
we need to prove that $e''\leq 2|A'|+|B'|-2$.
There are the following four cases to consider.
(The deletion of the two edges from $G$ is used only in the last
case, and is described there.)

\smallskip\noindent{1.}
If $|B'|\geq |A'|+2$, then by $3$-regularity of $G$,
$e'\leq 3|A'|\leq 2|A'|+|B'|-2$, and the result follows.

\smallskip\noindent{2.}
If $|B'|\leq |A'|-1$, then by $3$-regularity of $G$,
$e'\leq 3|B'|\leq 2|A'|+|B'|-2$, and the result follows.

\smallskip\noindent{3.}
If $|A'|=|B'|$, then by $3$-regularity and connectivity of $G$,
$e'\leq 3|A'| -1 =2|A'|+|B'|-1$, and so the only remaining subcase
here is the case of $G[V']$ being a $3$-regular
graph minus an edge. Then $G[V-V']$ is also a $3$-regular
graph minus an edge. Hence $G$ is the union of these two disjoint
induced subgraphs plus \emph{two} additional edges.
This however contradicts the fact that $G$ is a $3$-vertex
connected graph (indeed, $G$ is a graph of a simple
$3$-dimensional polytope), and hence also a $3$-edge connected graph.

\smallskip\noindent{4.}
 If $|B'|=|A'|+1$, then by $3$-regularity of $G$,
$e'\leq 3|A'|=2|A'|+|B'|-1$, and hence the only remaining subcase
here is the case of $e'= 3|A'|$. In this case
all neighbors of $A'$ are in $B'$, hence all neighbors of $B-B'$ are in
$A-A'$, which means that $G$ is the union of $G[V']$ and $G[V-V']$
plus three additional edges $e_1,e_2,e_3$ that connect $B'$ with $A-A'$.

On the level of our complex $K$, this means that $K$ is the union of
two $2$-dimensional subcomplexes $K'$ and $K''$, corresponding to the
graphs $G[V']$ and $G[V-V']$, respectively, and their common boundary,
$\partial(K')$, consists of $e_1$, $e_2$, and $e_3$. However, as
$\partial(K')$ is a union of cycles, it follows that the edges
$e_1,e_2,e_3$ form a cycle, and this cycle must not be a missing triangle
in $K$ by assumption. As both $A'$ and $B'$ are nonempty,
we infer that $K''$ is a single triangle, and as $|B'|=|A'|+1$,
this triangle belongs to side $A$.
%$e'=\frac{3}{2} f_2(K)-3 > \frac{3}{2} f_2(K)-4
%= 2|A'|+|B'|-2$, so it doesn't help!!
We can choose the edges $ab$ and $a'b'$ to be disjoint. Then not
both of them belong to $\partial(K')$, and so at least one of them
is in $G[V']$. Hence,
$e''\leq e'-1 \leq 2|A'|+|B'|-2$, and the result follows.
\end{proof}

%%%%%%%%%%%%%%%%%%%%%%%
%%%%%%%%%%%%%%%%%%%%%%%
\section{Preliminaries on simplicial complexes}\label{sec:PrelimComplexes}
First, we recall some basic definitions related to simplicial complexes,
to be used in Section \ref{sec:balancedKalai-Sarkaria}.
For further background see, for instance, \cite{Munkres}. Next, we motivate the
questions considered in Section \ref{sec:balancedKalai-Sarkaria}.

A \emph{simplicial complex} $K$ on the \emph{vertex set} $V$
is a collections of subsets of $V$ such that
(i) $\{v\}\in K$ for all $v\in V$, and
(ii) if $G\subset F$ and $F\in K$, then $G\in K$. The elements of $K$
are called {\em faces}. In particular, the empty set is a face of $K$.
A set $F\subseteq V$ that is not a face of $K$ but all of whose
proper subsets are faces of $K$ is called a \emph{missing face} of $K$.

For a simplicial complex $K$ and a face $\sigma$ of $K$, the
\emph{antistar} of $\sigma$ in $K$  is the subcomplex of $K$
given by $\antist_K(\sigma)=\{\tau\in K:\ \sigma\nsubseteq \tau\}$, and the
\emph{link} of $\sigma$ is the subcomplex $\lk_K(\sigma)=\{\tau\in K:\
\sigma\cap \tau=\emptyset, \ \sigma\cup \tau\in K\}$.
The \emph{join} of two simplicial complexes $K$ and $L$ on disjoint
vertex sets is $K*L=\{\sigma\cup\tau:\ \sigma\in K,\ \tau\in L\}$.
For instance, $[3]*[3]$ is simply $K_{3,3}$, where $[3]$
denotes the complex consisting of three isolated vertices.

The \emph{dimension} of a face $\sigma$ is defined by
$\dim(\sigma):=|\sigma|-1$; the dimension of a simplicial complex
$K$ is defined by $\dim(K):=\max\{ \dim(\sigma) \ : \ \sigma\in K\}$.
If all maximal (w.r.t.~containment) faces of $K$ have the same dimension,
then $K$ is \emph{pure}; the top-dimensional faces of $K$ are called
\emph{facets} and faces of codimension-$1$ are called \emph{ridges}.
For instance, the collection of all subsets of $[n]$
of size at most $i$ forms a pure $(i-1)$-dimensional
simplicial complex that we denote by $\binom{[n]}{\leq i}$.

If the vertices of $K$ can be colored by $\dim(K)+1$ colors
in such a way that the vertices of any edge of $K$ receive
different colors, then $K$ is \emph{balanced}.
For example, bipartite graphs are balanced $1$-dimensional complexes.
When discussing a balanced complex $K$, we assume that its vertex set $V$
is endowed with such a coloring: $V=V_1\uplus\cdots\uplus V_d$, where $d=\dim(K)+1$.
In this situation, for $T=\{i_1,\ldots,i_t\} \subseteq [d]$, we denote by $K_T$
the restriction of $K$ to the vertex set $V_{i_1}\uplus \cdots \uplus V_{i_t}$.

As in the case of graphs, for a simplicial complex $K$ on $V$ one can
define the \emph{Stanley-Reisner ring} of $K$. To do so, consider
a variable $x_v$ for every vertex $v\in V$. Let $S=\R[x_v \ :  v\in V]$,
let $I_K$ be the ideal of $S$ generated by squarefree monomials
corresponding to non-faces of $K$, and let $\R[K]:=S/I_K$.

Also, as in the case of bipartite graphs,
for a \emph{balanced} $(d-1)$-dimensional simplicial complex $K$ on
$V=V_1\uplus\cdots \uplus V_d$, where $V_i$ denotes the set of vertices of
color $i$, we can use the Stanley-Reisner ring of $K$ to define
a balanced shifting of $K$, $K^b$.
To do so, one needs a total order $<$ on $V$ that
extends given total orders on each of $V_1,\ldots,V_d$, as well as a block matrix
$\Theta=\Theta_1\times\cdots\times \Theta_d \in
\GL_{|V_1|}(\R)\times\cdots\times \GL_{|V_d|}(\R)$,
where $\Theta_1,\ldots,\Theta_d$ are generic matrices.
The rest of the definition is analogous to that for graphs:
for $v\in V_i$, set $\deg x_v:=e_i\in\Z^d$
(here $e_i$ is the $i$th unit vector); this makes $\R[K]$
into a $\Z^d$-graded ring. Now, for each $T\subseteq [d]$, let
$e_T=\sum_{i\in T}e_i\in\Z^d$,
and define $\B_T$ to be the greedy lexicographic basis (w.r.t.~$<$) of
the vector space $\R[K]_{e_T}$ chosen from the monomials written in
$\theta$'s. Let $\B=\cup_{T\subseteq [d]} \B_T$. Finally, define
$K^b$ as a collection of subsets of $V$ that are supports
of monomials from $\B$. It is shown in \cite{Babson-Novik} that
$K^b$ is a balanced simplicial complex; it has the same $f$-vector as $K$;
moreover, $K^b$ is balanced-shifted: if $v\in F\in K^b$ and
$w<v$ is a vertex of the same color as $v$, then
$F\setminus\{v\}\cup\{w\}\in K^b$.

We say that the order $<$ on $V=V_1\uplus \cdots \uplus V_d$
used to compute $K^b$ is \emph{$(l,l,\ldots,l)$-admissible}
if the least $l$ vertices from each colorset form an
initial segment of $<$.

Recall that by Euler's formula, a planar graph with $n\geq 3$ vertices
has at most $3n-6$ edges, and equality holds for the $1$-dimensional
skeleton of any triangulated $2$-sphere. Conjecture \ref{c:ksw} posits
an analogous statement for $2$-dimensional complexes embeddable in $\R^4$.
What happens in higher dimensions?

For a simplicial complex $K$, let $f_i(K)$ be the number of
$i$-dimensional faces (\emph{$i$-faces}) of $K$, and let $f(K)$
be the \emph{$f$-vector} of $K$,
 namely, $f(K):=(f_{-1}(K),f_0(K),\ldots,f_{\dim(K)}(K))$.
It follows from the Dehn-Sommerville relations \cite{Klee:DS}
and the generalized lower bound theorem \cite{St} that if $K$
is a $2d$-dimensional simplicial sphere that is the boundary
of a polytope, then $f_d(K)$ is linear in $f_{d-1}(K)$.
Is it true that for \emph{any}
$d$-dimensional complex $K$ topologically embeddable in the $2d$-sphere,
$f_d(K)$ is at most linear in $f_{d-1}(K)$? (That is, is there
some constant $c(d)$, depending only on $d$, such that $f_d(K)/f_{d-1}(K)\leq c(d)$
for all relevant $K$?) We consider this question
in the next section; we refer there to such inequality as
\emph{Euler-type upper bound} inequality.

%%%%%%%%%%%%%%%%%%%%%%%
%%%%%%%%%%%%%%%%%%%%%%%%
\section{Balanced complexes, Euler-type upper bounds, and the Kalai--Sarkaria
conjecture}\label{sec:balancedKalai-Sarkaria}

In this section we discuss a balanced analog of (a part of)
the Kalai--Sarkaria conjecture, potential applications of this conjecture
and possible approaches to attack it.
Our starting point is the following conjecture of
Kalai and Sarkaria \cite{Kalai-AlgShifting} that
implies McMullen's $g$-conjecture for simplicial spheres \cite{McMullen-g-conj}.
We let $C(d,n)$ denote the cyclic $d$-polytope with $n$ vertices, $\partial(C(d,n))$
stands for the boundary complex of $C(d,n)$, and $\S^d$ denotes the $d$-dimensional
sphere. Finally, for a simplicial complex $K$, $K^s$ denotes
the symmetric algebraic shifting of $K$.

\begin{conj}\label{conj:Kalai-Sarkaria}
Let $L$ be a simplicial complex with $n$ vertices.
If $L$ is topologically embeddable in $\S^{d-1}$, then
$L^s\subseteq (\partial(C(d,n)))^s$.
In particular, if $K$ is a $d$-dimensional complex embeddable
in $\S^{2d}$, then $K^s$ does not contain the Flores
complex $\binom{[2d+3]}{\leq d+1}$.
\end{conj}

We posit the following bipartite analog of the ``in particular'' part:

\begin{conj}\label{conj:balancedKalaiSarkaria}
Let $K$ be a $d$-dimensional balanced complex that is topologically embeddable
in $\S^{2d}$, and let $<$ be a $(2,2,\ldots, 2)$-admissible order.
Then $K^{b, <}$ does not contain the
van Kampen complex $[3]^{*(d+1)}$, i.e., the
$(d+1)$-fold join of $3$ points.\footnote{The
statements of Conjectures \ref{conj:Kalai-Sarkaria}, \ref{2-dimKS},
and \ref{conj:balancedKalaiSarkaria} are also conjectured
to hold for the case of exterior shifting (balanced exterior shifting, resp.).
In fact, in an unpublished work, Nevo established the exterior shifting
counterpart of Proposition \ref{prop:K_5}.}
\end{conj}

As with the Kalai--Sarkaria conjecture,
Conjecture \ref{conj:balancedKalaiSarkaria} is known
so far only for $d=0,1$: the case $d=0$ is obvious
and the case $d=1$ is Theorem \ref{theo:BipartiteGluck}.
We observe that Conjecture \ref{conj:balancedKalaiSarkaria}
implies a weaker version of Conjecture \ref{conj:Kalai-Sarkaria}
concerning Euler-type upper bound inequalities
(see \cite{Grunbaum-70}) for {\em all}
simplicial complexes (cf.~ Conjecture \ref{c:ksw}):

\begin{prop}\label{prop:Heawood-ineq}
If Conjecture \ref{conj:balancedKalaiSarkaria} is true,
then for every nonnegative integer $d$ the following holds:
\begin{enumerate}
\item[(i)] If $\Gamma$ is a
$d$-dimensional balanced complex that embeds in $\S^{2d}$,
then \[f_d(\Gamma)\leq 2 f_{d-1}(\Gamma).\]

\item[(ii)] There exists a constant $c(d)$ such that for
an arbitrary $d$-dimensional simplicial complex $K$ that embeds in
$\S^{2d}$, $f_d(K)\leq c(d) f_{d-1}(K)$.
\end{enumerate}
\end{prop}

\begin{proof}
(i) It follows from Conjecture \ref{conj:balancedKalaiSarkaria} that
for any facet $F$ in $\Gamma^b$ there must be a colorset $V_i$
such that $F$ contains one of the two \emph{minimal} elements of $V_i$.
Since the total order $<$ on $V$ is $(2,2,\ldots,2)$-admissible,
we conclude that
the map $F\mapsto F\setminus\{\min_<(F)\}$ from the set of facets
of $\Gamma^b$ to the set of $(d-1)$-faces of $\Gamma^b$, is at most
$2 : 1$. The fact that balanced shifting preserves $f$-vectors
then yields $f_d(\Gamma)\leq 2 f_{d-1}(\Gamma)$.

(ii)
In a random coloring of the vertices of $K$ by $d+1$ colors,
the probability that a given facet is colorful (i.e., contains a vertex of
each color) is $\frac{(d+1)!}{(d+1)^{d+1}}$.
Thus, there is a coloring with at least $\frac{(d+1)!}{(d+1)^{d+1}}f_d(K)$
colorful facets; denote by $L$
the balanced subcomplex of $K$ spanned by these facets.
Then by part (i),
$$f_d(K)\leq \frac{(d+1)^{d+1}}{(d+1)!}f_d(L)\leq
\frac{(d+1)^{d+1}}{(d+1)!} 2 f_{d-1}(L)
\leq \frac{2(d+1)^{d+1}}{(d+1)!}f_{d-1}(K).$$
Hence, taking $c(d)=2(d+1)^{d+1}/(d+1)!$ completes the proof.
\end{proof}

\noindent We remark that Conjecture~\ref{conj:Kalai-Sarkaria}, if true,
would imply that $c(d)=d+2$, while from Conjecture~\ref{conj:balancedKalaiSarkaria}
we only derived the weaker estimate of $c(d)=2(d+1)^{d+1}/(d+1)!$.

We next show that the above Euler-type inequality implies a weaker version of
Conjecture \ref{conj:balancedKalaiSarkaria}, so ``up to constants''
they are equivalent; more precisely:

\begin{prop}\label{prop:Heawood-->balancedKalaiSarkaria}
Assume there is a constant $c(d)$ such that for every $d$-dimensional
balanced complex $K$ embeddable in $\S^{2d}$,
$f_d(K)\leq c(d) f_{d-1}(K)$. Let
$C(d)=(d+1)c(d)$. Then for every $d$-dimensional balanced
complex $K$ embeddable in $\S^{2d}$ and a
$(C(d),\ldots,C(d))$-admissible order $<$,
$K^{b,<}$ does not contain $[C(d)+1]^{*(d+1)}$.
\end{prop}
\begin{proof}
Our assumption that $f_d(K)\leq c(d) f_{d-1}(K)$ implies that there is
a ridge of $K$ that is contained in at
most $(d+1)c(d)$ facets of $K$. Now use the high-dimensional
Deletion Lemma (see Lemma \ref{lem:deletion--HighDim} below) and induction.
\end{proof}

We now turn to rephrasing Conjecture \ref{conj:balancedKalaiSarkaria}
in terms of embeddability of $K^b$, %the \emph{balanced-shifted} complex,
a formulation that is \emph{not} available for Conjecture \ref{conj:Kalai-Sarkaria}:
indeed, while shifted graphs not containing $K_5$ may be nonplanar
(for instance, $G^s$ where $G$ is the graph of the octahedron),
\emph{balanced-shifted} bipartite graphs not containing $K_{3,3}$
are necessarily planar. This statement extends to higher dimensions,
as the following proposition shows.

\begin{prop}\label{prop:EmbedNo[3]*d}
Let $K$ be a $d$-dimensional balanced-shifted simplicial complex
not containing $[3]^{*(d+1)}$ as a subcomplex. Then $K$ is PL embeddable
in $\S^{2d}$.
\end{prop}
\begin{proof}
Among all  $d$-dimensional balanced-shifted
simplicial complexes on the same vertex set as $K$, let
$\Gamma(d)$ be the maximal complex not containing $[3]^{*(d+1)}$.
In other words, the facets of
$\Gamma(d)$ are all the colorful $(d+1)$-subsets of $V$
that contain one of the least two vertices of some color.
We need to show that $\Gamma(d)$  is PL embeddable
in $\S^{2d}$.

For $d=0$ this is clear, as $\Gamma(d)$ consists of two points.
For $d=1$, this is also easy: in the plane, draw a square with vertices
$1,1',2,2'$; embed the vertices $3,4,\ldots,n$ in the open segment
connecting $1$ and $2$, and  the vertices $3',4',\ldots,m'$ in the parts
of the straight line through $1'$ and $2'$ that lie outside of the square;
now draw as straight segments the edges $ij'$ where at least one of $i,j\leq 2$.

We show by induction on $d$ how to PL embed $\Gamma(d)$ in $\mathbb{R}^{2d}$ for $d>1$.
Consider the first $d$ (out of $d+1$) colorsets of $\Gamma(d)$,
and two subcomplexes of $\Gamma(d)$ on these colors: $\Gamma(d-1)$ and $\Gamma(d)_{[d]}$.
(Note that $\Gamma(d-1)\subseteq \Gamma(d)_{[d]}$.)
Assume that $\Gamma(d-1)$ is PL embedded in $\mathbb{R}^{2d-2}\times\{0\}\times\{0\}$.
As $\dim \Gamma(d-1)=\dim \Gamma(d)_{[d]}=d-1$, we can extend this embedding to a PL map
from $\Gamma(d)_{[d]}$  into $\mathbb{R}^{2d-2}\times\{0\}\times\{0\}$ in such a way that
(i) the only intersections occur between pairs of facets that involve at least one of the
``added'' faces (i.e., faces of $\Gamma(d)_{[d]}$ that do not belong to $\Gamma(d-1)$),
(ii) they occur at interior points, and (iii) there are finitely many such points.
Now resolve these intersections by pulling the added
$(d-1)$-faces, one by one, into the \emph{negative} side of the last coordinate
(keeping the coordinate before last equal to zero).
Figure 1 illustrates the case of $d=2$, $n=4$, $m'=3'$.

\begin{figure}[ht]\label{pic}
\begin{center}
\includegraphics[height=7cm]{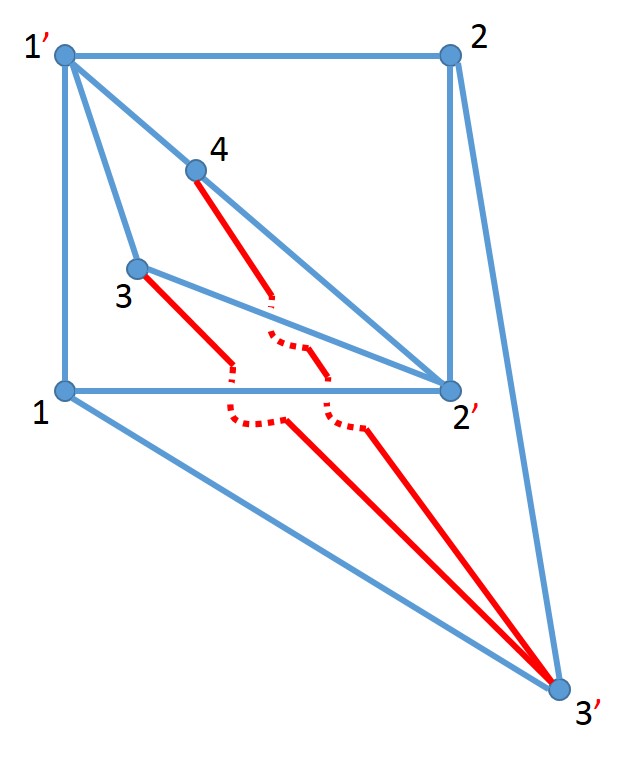}
\end{center}
\caption{The first step of the case $d=2$ (with $n=4, m=3$).
 The solid lines are the edges of $\Gamma(1)$;
 the partially dashed lines are the edges of $\Gamma(2)_{[2]}-\Gamma(1)$.}
\end{figure}

Next, place the first and second vertices of color $d+1$ at $\pm v$,
where $v$ is the unit vector with the coordinate before last equal to $1$,
and consider two geometric cones over the above embedding of $\Gamma(d)_{[d]}$:
one with apex $v$ and another one with apex $-v$.
The union of these cones provides an embedding of the suspension of $\Gamma(d)_{[d]}$,
$\Sigma(\Gamma(d)_{[d]})$,  and this embedding is such that
the last coordinate is always nonpositive.

Finally, place the remaining vertices (i.e., vertices number $3,4,\ldots$)
of color $d+1$ at distinct points on the open arc
$\{(0,\cdots,0, t,s):\ t^2+s^2=1,\ s,t>0\} \subset \R^{2d-2}\times \R^2$,
and for each of those points, construct a geometric cone over  $\Gamma(d-1)$ with that
point as the apex. These cones lie in distinct half hyperplanes with a common boundary,
namely $\mathbb{R}^{2d-2}\times\{0\}\times\{0\}$, and hence this union of cones is embedded.
All the new points added in this step have a \emph{positive} last coordinate,
and thus are disjoint from the embedding of $\Sigma(\Gamma(d)_{[d]})$.
Together with that embedding of $\Sigma(\Gamma(d)_{[d]})$,
they form an embedding of $\Gamma(d)$ into $\mathbb{R}^{2d}$.
\end{proof}

Combining Proposition \ref{prop:EmbedNo[3]*d} with the well-known fact
that the complex $[3]^{*(d+1)}$ is not PL embeddable in $\S^{2d}$ \cite{VanKampen,Flores},
we obtain that Conjecture \ref{conj:balancedKalaiSarkaria} is equivalent to the following:

\begin{conj}\label{conj:eqBalancedKalaiSarkaria}
If $K$ is a $d$-dimensional balanced complex that is topologically embeddable
in $\S^{2d}$, and $<$ is a $(2,\ldots,2)$-admissible order,
then $K^{b,<}$ is PL embeddable in $\S^{2d}$.
\end{conj}

Let $\mathfrak{o}(K)$ denote the van Kampen obstruction to PL
embeddability of a $d$-dimensional complex $K$ in $\S^{2d}$,
computed with coefficients in $\mathbb{Z}$. (One may also
use other coefficients, e.g., $\Z/2\Z$). Recall that if $K$ is
PL embeddable in $\S^{2d}$, then $\mathfrak{o}(K)=0$
(and the converse also holds provided $d\neq 2$),
see \cite{Shapiro,Wu,Freedman-Krushal-Teichner}. As
$\mathfrak{o}([3]^{*(d+1)})\neq 0$ (even with $\mathbb{Z}_2$
coefficients) and as, according to \cite{Bry72}, for $d\geq 3$
the topological embeddability of a $d$-dimensional complex $K$ in
$\S^{2d}$ is equivalent to the PL embeddability, we
obtain that for $d\geq 3$, the following conjecture implies
Conjecture \ref{conj:eqBalancedKalaiSarkaria},
even when considered with $\mathbb{Z}_2$ coefficients.

\begin{conj}\label{conj:vanKampenBalancedKalaiSarkaria}
Let $K$ be a $d$-dimensional balanced complex. If $\mathfrak{o}(K)=0$,
then $\mathfrak{o}(K^b)=0$.
\end{conj}

%%%%%%%%%%
We are now in a position to introduce a balanced rigidity matrix
corresponding to Conjecture~\ref{conj:balancedKalaiSarkaria}.
As with bipartite graphs, for a $d$-dimensional balanced complex $K$
and a fixed integer $l$, assign to each vertex $v\in K$ a generic
$l$-dimensional real vector $\theta_v$, and define the
following facet-ridge matrix $M(K,l)$: the rows
of $M(K,l)$ correspond to the facets $F$ of $K$, the columns of $M(K,l)$ come in
$l$-tuples with each $l$-tuple corresponding to a ridge $G$ of $K$;
the $1\times l$ block of $M(K)_{F,G}$ is $(0)$ if $G\nsubseteq F$ and
$\theta_{F-G}$ otherwise. (Thus, if $K$ is $1$-dimensional then
$M(K,l) = R^{(l,l)}(K)$.) Arguing as in
Lemma \ref{mapPhi} (and Proposition \ref{rigid/stressfree-criterion}),
we obtain:

\begin{lem}  \label{does_not_contain}
 For an $(l,\ldots, l)$-admissible order $<$, the
complex $K^{b,<}$ does not contain $[l+1]^{*(d+1)}$ as a subcomplex
if and only if the rows of the matrix $M(K,l)$ are linearly independent.
\end{lem}

The following is a high-dimensional analog of the Deletion
Lemma (Lemma \ref{lem:bipDeletion}).

\begin{lem}\label{lem:deletion--HighDim}
Let $K$ be a $d$-dimensional balanced complex,
$<$ an $(l,\ldots,l)$-admissible order,
$G$ a ridge of $K$ contained in at most $l$ facets of $K$,
 and $L=\antist_K(G)$. If $[l+1]^{*(d+1)}\nsubseteq L^{b,<}$ then
$[l+1]^{*(d+1)}\nsubseteq K^{b,<}$.
\end{lem}
\begin{proof}
Let $s\leq l$ denote the number of facets of $K$ that contain $G$.
The matrix $M(K,l)$ is obtained from $M(L,l)$ by adjoining $l$ columns
corresponding to $G$ and $s$ rows corresponding to the
facets containing $G$. These new $l$ columns consist of zeros followed by
a generic $s\times l$ block (at the intersection with the new $s$ rows).
Thus $\rank M(K,l)\geq\rank M(L,l)+s$. Since by our assumption on $L$,
$\rank M(L,l)$ equals the number of rows of $M(L,l)$, the quantity
$\rank M(L,l)+s$ coincides with the number of rows of $M(K,l)$.
Hence, the rows of $M(K,l)$ are linearly independent.
\end{proof}

In the rest of the section, we gather some evidence in favor of
Conjecture \ref{conj:balancedKalaiSarkaria}. Specifically, we consider
certain basic constructions on balanced simplicial complexes and their effect
on balanced shifting. We start with the join operation. All balanced shiftings in
the rest of this section are computed w.r.t.~$(2,\ldots,2)$-admissible orders.

\begin{lem}\label{lem:join}
Let $K$ be a $k$-dimensional balanced complex embeddable in $\S^{2k}$,
and let $L$ be any $l$-dimensional balanced complex.
Then $K*L$ is a $(k+l+1)$-dimensional balanced complex embeddable in
the $2(k+l+1)$-sphere. Moreover, if $K$ satisfies the conclusion of
Conjecture \ref{conj:balancedKalaiSarkaria}, then so does $K*L$.
\end{lem}
\begin{proof}
As any $l$-dimensional simplicial complex
embeds in the $(2l+1)$-sphere, our assumption on $K$ implies that
$K*L$ embeds in the $2(k+l+1)$-sphere.
Assume that $K^b$ and $L^b$ are computed w.r.t.~linear
$(2,\ldots,2)$-admissible orders $<_K$ and $<_L$, respectively,
and that $(K*L)^b$ is computed
w.r.t.~a linear order $<$ that extends the partial order $<_K \uplus <_L$.
It then follows from the definition of the balanced shifting that
$(K*L)^b=K^b * L^b$. Thus, if $K^b$ does not contain $[3]^{*(k+1)}$, then
 $K^b*L^b$ does not contain $[3]^{*(k+l+2)}$  as it does not
even contain its subcomplex $[3]^{*(k+1)}$ on the first $k+1$ colors.
\end{proof}

Next we consider the effect of certain subdivisions. To do so,
for a balanced complex $L$, we use the balanced rigidity matrix
$M(L):=M(L,2)$.

Let $K$ be a pure balanced $d$-dimensional complex and $\sigma$ a face of $K$
that is not a vertex. Let $S$ be any pure balanced complex of the same dimension as
$\sigma$ and assume that $S$ has a missing facet $X$.
Identify the vertices of this missing facet with the correspondingly colored vertices
of $\sigma$ and define \[K'=\antist_K(\sigma) \cup (S*\lk_K(\sigma)).\]
In other words, $K'$ is obtained from $K$ by removing the star of $\sigma$,
$\overline{\sigma} * \lk_K(\sigma)$, and replacing it with $S*\lk_K(\sigma)$.
Then $K'$ is a pure balanced $d$-dimensional complex.
Further, if $S$ is a ball whose boundary coincides with that of $X$ (i.e,
$S$ is obtained from a balanced sphere by removing one facet, $X$),
then $K'$ is homeomorphic to $K$.

\begin{prop}\label{prop:sdLink}
If $(\antist_K(\sigma))^b$ does not contain $[3]^{*(d+1)}$ and $(\lk_K(\sigma))^b$
does not contain $[3]^{*(d-|\sigma|+1)}$, then $(K')^b$ does not contain
$[3]^{*(d+1)}$.
\end{prop}
\begin{proof}
Denote the facets of $S$ by $H_1,\ldots, H_m$. The facets of $K'$ then fall in two
categories: (a) the facets of $\antist_K(\sigma)$, and (b) for each
$i=1,\ldots, m$, the facets of the form $G\cup H_i$ where $G$ is a facet of
$\lk_K(\sigma)$; we denote this $i$-th set of facets by $\F_i$. Similarly,
the ridges of $K'$ are of the following types:
(a) the ridges of $\antist_K(\sigma)$,
(b) for each $i=1,\ldots, m$, the ridges of the form $R\cup H_i$ where $R$ is a
ridge of $\lk_K(\sigma)$ ---  we denote this $i$-th set of ridges by $\RR_i$, and
(c) all remaining ridges.  In the following we will ignore the ridges of type (c);
specifically, we will show that the restriction of $M(K')$ to the columns of
the ridges of types (a) and (b) already has independent rows.

Consider the balanced rigidity matrix $M(K')$. Its restriction to facets/ridges
of $\antist_K(\sigma)$ coincides with the restriction of $M(K)$ to the same rows and
columns, and thus, by our assumption on $\antist_K(\sigma)$, has rank
$f_d(\antist_K(\sigma))$ (see Lemma \ref{does_not_contain}).
For each $i=1,\ldots, m$, the restriction of
 $M(K')$ to the columns labeled by the ridges from $\RR_i$ consists of the block
$M(\lk_K(\sigma))$ positioned at the intersection with the rows labeled by
the elements of $\F_i$, and zeros everywhere else. By our assumption on the
link, the rank of such a block equals the number of facets of the link.
As all these blocks have pairwise disjoint sets of columns and rows,
it follows that
\begin{eqnarray*}
\rank(M(K')) &\geq& \rank(M(\antist_K(\sigma))+ m \cdot \rank(M(\lk_K(\sigma))\\
&=& f_d(\antist_K(\sigma))+ f_{|\sigma|-1}(S)
\cdot f_{d-|\sigma|}(\lk_K(\sigma))=f_d(K').
\end{eqnarray*}
Thus the above inequality is, in fact, equality, and
$(K')^b$ does not contain $[3]^{*(d+1)}$.
\end{proof}

Finally, we show that if $S$ is obtained from
a $(|\sigma|-1)$-dimensional balanced pseudomanifold
 by removing one facet, $X$, then the condition on the link in
Proposition \ref{prop:sdLink} can be dropped. More generally:

\begin{prop}\label{prop:sdPseudomanifold}
Let $K$ be a pure balanced $d$-dimensional complex and $\sigma$ a face
of $K$ that is not a vertex. Let $S$
%be obtained from a $(|\sigma|-1)$-dimensional balanced
%pseudomanifold (without boundary) by removing one facet, $X$,
be a pure $(|\sigma|-1)$-dimensional balanced simplicial complex
with a missing facet $X$ and such that
each ridge of $S$ is in at most two facets.
Let $K'=\antist_K(\sigma) \cup (S*\lk_K(\sigma))$.
 If $(\antist_K(\sigma))^b$ does not contain $[3]^{*(d+1)}$,
then $(K')^b$ does not contain $[3]^{*(d+1)}$.
\end{prop}
\begin{proof}
Delete from the matrix $M(S)$ all columns corresponding to the ridges that are
subsets of $X$; denote the resulting matrix by $M^*(S)$.
Note that $M^*(S)$ has no zero rows: this is because every facet of $S$ has at least
one ridge that is not a subset of $X$. Moreover, every ridge is in at most two
facets, and so the rows of $M^*(S)$ are linearly independent (the same argument as in
the proof of Lemma \ref{lem:deletion--HighDim} applies.)

Now let the facets of $\lk_K(\sigma)$ be $H_1,\ldots, H_k$. The set of facets
of $K'$ consists of (a) the facets of $\antist_K(\sigma)$, and (b) for each
$i=1,\ldots, k$, the facets of the form $G\cup H_i$, where $G$ is a facet of $S$;
denote this $i$-th set of facets by $\F_i$. The set of ridges of $K'$ consists
of (a) the ridges of $\antist_K(\sigma)$, (b) for each $i=1,\ldots, k$, the ridges of
the form $R\cup H_i$ where $R$ is a ridge of $S$ not contained in $X$ --- denote
this $i$-th set of ridges by $\RR_i$, (c) all remaining ridges, which we will
ignore.

We again consider the balanced rigidity matrix of $K'$.
As in Proposition \ref{prop:sdLink}, the restriction of $M(K')$ to
facets/ridges of $\antist_K(\sigma)$ has rank $f_d(\antist_K(\sigma))$.
For each $i=1,\ldots, k$, the restriction of $M(K')$ to the columns labeled by
the ridges from $\RR_i$ consists of the block $M^*(S)$
(positioned at the intersection with the rows labeled by the elements of
$\F_i$) and zeros everywhere else. By the first paragraph of this proof,
the rank of such a block equals the number of facets of $S$. As all these
blocks have pairwise disjoint sets of columns and rows, we obtain that
\begin{eqnarray*}
\rank(M(K')) &\geq& \rank(\antist_K(\sigma))+ k \cdot \rank(M^*(S))\\
&=& f_d(\antist_K(\sigma))+ f_{d-|\sigma|}(\lk_K(\sigma))
\cdot f_{|\sigma|-1}(S)=f_d(K').
\end{eqnarray*}
The result follows.
\end{proof}

We conclude with a conjecture on linklessly embedable complexes.
A high-dimensional analog of Sachs' result \cite{Sachs83} on
linkless embeddability is due to Skopenkov \cite[Lemma 1]{Skopenkov}.
It asserts that $[4]^{*(d+1)}$ is not linklessly embeddable in  $\R^{2d+1}$.
This theorem leads us to pose the
following generalization of Conjecture \ref{conj:bipLinklessGraph},
analogous to Conjecture \ref{conj:balancedKalaiSarkaria}.

\begin{conj}\label{conj:linkless-shifting}
Let $K$ be a $d$-dimensional balanced simplicial complex that is linklessly
embeddable in $\mathbb{R}^{2d+1}$ and let $<$ be a $(3,\ldots,3)$-admissible
order. Then $K^{b,<}$ does not contain $[4]^{*(d+1)}$.
\end{conj}

%%%%%%%%%%%%%%%%%%%%%%%%%%%%
\section*{Acknowledgments}
We thank Maria Chudnovsky and Paul Seymour for helpful discussions, and
Amit Singer for bringing \cite{SingerCucuringu} to our attention.
We are also grateful to the referee for a very thorough reading of
our paper and many insightful remarks.

\bibliography{gbiblio}
\bibliographystyle{plain}
\end{document}